\documentclass[10pt]{amsart}

\usepackage{amsmath,amssymb,mathrsfs,color}

\allowdisplaybreaks
\numberwithin{equation}{section}

\newtheorem{lemma}{Lemma}[section]
\newtheorem{theorem}[lemma]{Theorem}
\newtheorem{remark}[lemma]{Remark}
\newtheorem{proposition}[lemma]{Proposition}
\newtheorem{coro}[lemma]{Corollary}
\newtheorem{definition}[lemma]{Definition}
\newtheorem{example}[lemma]{Example}

\title[Poisson stable motions of monotone NDS]
{Poisson stable motions of monotone nonautonomous dynamical systems}

\author{David Cheban}
\address{D. Cheban\footnote{Permanent address: State University of Moldova, Faculty of Mathematics and
Informatics, Department of Mathematics, A. Mateevich Street 60, MD--2009 Chi\c{s}in\u{a}u, Moldova}: School of Mathematical Sciences,
Dalian University of Technology, Dalian 116024, P. R. China}
\email{cheban@usm.md; davidcheban@yahoo.com}

\author{Zhenxin Liu}
\address{Z. Liu: School of Mathematical Sciences,
Dalian University of Technology, Dalian 116024, P. R. China}
\email{zxliu@dlut.edu.cn}

\date{November 25, 2017}

\begin{document}
\begin{abstract}
In this paper, we study the Poisson stability
(in particular, stationarity, periodicity, quasi-periodicity, Bohr almost
periodicity, almost automorphy, recurrence in the sense of Birkhoff,
Levitan almost periodicity, pseudo
periodicity, almost recurrence in the sense of Bebutov, pseudo recurrence, Poisson stability) of
motions for monotone nonautonomous dynamical systems and of solutions
for some classes of monotone nonautonomous evolution equations (ODEs,
FDEs and parabolic PDEs). As a byproduct, some of our results
indicate that all the trajectories of monotone systems converge to the above mentioned
Poisson stable trajectories under some suitable conditions, which is interesting
in its own right for monotone dynamics.
\end{abstract}

\maketitle

\section{Introduction}\label{S1}

The existence of Bohr almost periodic solutions of the equation
\begin{equation}\label{eq12}
x'=f(t,x)
\end{equation}
with Bohr almost periodic right-hand side $f$  in $t$, uniformly
with respect to (shortly w.r.t.) $x$ on every compact subset of
$\mathbb R^n$ was studied by many authors
\cite{JZ_2005,NOS_2007,Opi_1959,SaSe_1977,Shch_1974,SYY,Zhi69} (see
also \cite[ChIV]{bro75}, \cite[ChXII]{Fin}, \cite[ChVII]{Lev-Zhi} and
the bibliography therein).

Z. Opial \cite{Opi_1959} consider the scalar (i.e. $n=1$) case of differential equation
(\ref{eq12}). If $f$ is monotone w.r.t. spacial variable $x$,
he established that every bounded (on the whole axis $\mathbb R$) solution is Bohr almost periodic.

Recall that a function $f\in C(\mathbb R\times\mathbb R^n,\mathbb
R^n)$ is called {\em regular} if for any $g\in H(f)$ and $v\in\mathbb
R^{n}$ the limit equation
\begin{equation}\label{eq12g}
x'=g(t,x)
\end{equation}
admits a unique solution $\varphi(t,v,g)$ defined on $\mathbb R$ with initial value $x(0)=v$,
where $H(f):=\overline{\{f^{\tau}:\ \tau \in
\mathbb R\}}$, $f^{\tau}(t,x):=f(t+\tau,x)$ for any $(t,x)\in\mathbb
R^n$ and by bar we mean the closure in $C(\mathbb R\times \mathbb
R^n,\mathbb R^n)$ which is equipped with the compact-open topology.

V. V. Zhikov \cite{Zhi69} (see also \cite[ChVII]{Lev-Zhi} and
\cite[ChIV]{bro75}) studied the scalar equation (\ref{eq12}) with regular $f$
but without monotone assumption for $f$. He
obtained existence of at least one almost periodic solution of
equation (\ref{eq12}) if it admits one bounded and uniformly stable solution.

R. Sacker and G. Sell \cite{SaSe_1973}  (see also \cite[\S 3.6]{SaSe_1977} and \cite[ChXII]{Fin}) generalized V. V. Zhikov's result,
still for scalar equations, by
replacing the regularity of $f$ by positive regularity. Namely: a function $f\in C(\mathbb R\times\mathbb
R^n,\mathbb R^n)$ is called {\em positively regular} if for any $g\in
H(f)$ and $v\in\mathbb R^{n}$ the limit equation (\ref{eq12g}) admits
a unique solution $\varphi(t,v,g)$ defined on $\mathbb R_{+}$ with initial
value $x(0)=v$. In addition, R. Sacker and G. Sell \cite{SaSe_1977} studied
almost periodic solutions for general case (not necessarily scalar) of equation \eqref{eq12}
in the framework of skew-product flows.

B. A. Shcherbakov \cite{Shch_1974} studied the Poisson stability (in particular, periodicity, Bohr almost periodicity,
recurrence in the sense of Birkhoff, almost recurrence in the sense of
Bebutov, Levitan almost periodicity) of solutions for scalar equation
(\ref{eq12}) with $f$ being monotone w.r.t. $x$ and Poison stable in $t\in \mathbb R$ (uniformly
w.r.t. $x$ on every compact subset of $\mathbb R$). That is, he generalized
Z. Opial's result to Poisson stable differential equations (\ref{eq12}).

D. Cheban \cite{Che_pre} considered the
Poisson stable solutions for the scalar equation (\ref{eq12}) with
arbitrary Poisson stable (w.r.t. time $t$) $f$ and without
monotonicity assumption for $f$, thus generalized the results of
Z. Opial, V. V. Zhikov, R. Sacker and G. Sell, and B. A.
Shcherbakov.

In the framework of monotone cocycles or nonautonomous dynamical
systems the problem of almost periodicity and almost automorphy of
solutions for equation (\ref{eq12}) in general case (both finite
and infinite dimensional cases) was studied in the works of W. Shen and Y. Yi \cite{SYY}, J.
Jiang and X.-Q. Zhao \cite{JZ_2005}, S. Novo {\it et al} \cite{NOS_2007}, and the bibliography therein.

The aim of this paper is to study the existence of Poisson stable (e.g. stationary, periodic, quasi-periodic,
Bohr/Levitan almost periodic, almost automorphic,
almost recurrent, etc) solutions of equation (\ref{eq12}) in both finite
and infinite dimensional cases when (\ref{eq12}) generates a monotone cocycle, which can be
achieved, say, if $f$ is quasi-monotone.
The existence of at least one such Poisson stable solution is obtained provided
each solution of (\ref{eq12}) is compact on $\mathbb R_{+}$ and uniformly stable.
Meantime, 
as a byproduct our results (see Theorems \ref{thM1_1}, \ref{com} and \ref{Lya}) show that all the solutions
will converge to the Poisson stable ones, which is interesting on its own rights in monotone dynamics.

The paper is organized as follows.

In Section 2 we collect some notions and facts from the theory of
nonautonomous dynamical systems: cocycles, nonautonomous dynamical systems, conditional compactness,
Ellis semigroup etc.

Section 3 is dedicated to studying the structure of the
$\omega$-limit set of noncompact semi-trajectories for autonomous and
nonautonomous dynamical systems. The main result of this section is
contained in Theorem \ref{th02} which states that the one-sided
dynamical system $(X,\mathbb R_{+},\pi)$ on $\omega_{x_0}$, the
$\omega$-limit set of $x_0$, can be extended to a
two-sided dynamical system provided the positive semi-trajectory
$\Sigma_{x_0}^{+}$ of $x_0$ is conditionally precompact and $\omega_{x_0}$ is uniformly stable.

In Section 4 we give a survey of different classes of Poisson
stable motions, B. A. Shcherbakov's principle of comparability of
motions by their character of recurrence and some generalization of
this principle.

Section 5 is dedicated to the study of abstract monotone
nonautonomous dynamical systems. The main results of the paper
are contained in Theorems \ref{thM1} and \ref{thM2} which give sufficient conditions
for existence of comparable and strongly comparable
motions. Using these Theorems and Shcherbakov's comparability principle of motions
by character of recurrence we obtain a series
of results of existence of stationary (respectively, periodic, quasi-periodic,
Bohr almost periodic, almost automorphic, Birkhoff recurrent, Levitan almost
periodic, almost recurrent, pseudo recurrent, uniformly Poisson
stable, Poisson stable) motions. As mentioned above,
we also obtain the convergence of all the trajectories to Poisson stable ones
under some suitable conditions.

In Section 6 we apply our abstract results obtained in Sections 3 and 5
to study different classes of Poisson stability (as listed above) of solutions
for monotone differential equations (ODEs,
FDEs and parabolic PDEs). In this way we obtain a series of new
results (some of them coincides with the well-known results).

\section{NDS: some general properties}\label{S2}

In this section we collect some notions and facts for
nonautonomous dynamical systems which we will use below;
the reader may refer to \cite{Che_2001}, \cite[Ch. IX]{Che_2015}, \cite{Sel_71} for details.

Throughout the paper, we assume that $X$ and $Y$ are metric spaces and for simplicity we use the same notation $\rho$
to denote the metrics on them, which we think would not lead to confusion.
Let $\mathbb R=(-\infty,+\infty)$, $\mathbb R _{+}=\{t\in \mathbb R: t \ge
0 \}$ and $\mathbb R_{-}=\{t \in \mathbb R: t \le 0 \}$.
For given dynamical system $(X,\mathbb R, \pi)$ and given point $x\in X$, we denote by $\Sigma_x$ (respectively, $\Sigma_x^+$)
its {\em trajectory} (respectively, {\em semi-trajectory}),
i.e. $\Sigma_x:=\{\pi(t,x): t\in\mathbb R\}$ (respectively, $\Sigma_x^+:=\{\pi(t,x): t\in\mathbb R_+\}$),
and call the mapping $\pi(\cdot,x):\mathbb R\to X$ the {\em motion} through $x$ at the moment $t=0$.
For given set $A\subseteq X$, we denote $\Sigma_A:=\{\pi(t,x): t\in\mathbb R, x\in A\}$; $\Sigma_A^+$ is defined similarly.
We denote the {\em hull} (respectively, {\em semi-hull}) of a point $x$ by $H(x):=\overline\Sigma_x$
(respectively, $H^+(x):=\overline\Sigma_x^+$), where by bar we mean closure.
A point $x\in X$ is called {\em Lagrange stable}, ``st. $L$" in short, (respectively,
{\em positively Lagrange stable}, ``st. $L^{+}$" in short)
if $H(x)$ (respectively, $H^{+}(x)$) is compact.

\subsection{Cocycles and NDS}

Let $(Y ,\mathbb R, \sigma )$ be a two-sided dynamical
system on $Y $ and $E$ a metric space.

\begin{definition}\label{def2.1}\rm 
A triplet $\langle E, \phi ,(Y ,
\mathbb R, \sigma )\rangle$ (or  briefly $\phi$ if no confusion) is said
to be a {\em cocycle} on state space (or fibre) $E$ with base $(Y,\mathbb
R,\sigma )$ if the mapping $\phi : \mathbb R _{+} \times  Y \times E
\to E $ satisfies the following conditions:
\begin{enumerate}
\item $\phi (0,u,y)=u $ for all $u\in E$ and $y\in Y$;
\item $\phi (t+\tau,u,y)=\phi (t, \phi (\tau,u, y), \sigma(\tau,y))$
for all $ t, \tau \in \mathbb R _{+},u \in E$ and $ y \in Y$;
\item the mapping $\phi $ is continuous.
\end{enumerate}
\end{definition}

\begin{definition}\label{def2.2} \rm 
Let $ \langle E ,
\phi,(Y,\mathbb R,\sigma)\rangle $ be a cocycle on $E$, $X:=E\times
Y$ and $\pi $ be a mapping from $ \mathbb R _{+} \times X $ to $X$
defined by $\pi :=(\phi ,\sigma)$, i.e. $\pi
(t,(u,y))=(\phi (t,u,y),\sigma(t,y))$ for all $ t\in \mathbb R
_{+}$ and $(u,y)\in E\times Y$. The triplet $ (X,\mathbb R _{+},
\pi)$ is an autonomous dynamical system and called
{\em skew-product dynamical system}.
\end{definition}

\begin{definition}\label{def2.3} \rm 
Let $\mathbb T_{1}\subseteq \mathbb T _{2} $ be two subsemigroups of the
group $\mathbb R$, $(X,\mathbb T _{1},\pi )$ and $(Y ,\mathbb T _{2},
\sigma )$ be two autonomous dynamical systems and $h: X \to Y$ be a
homomorphism from $(X,\mathbb T _{1},\pi )$ to $(Y , \mathbb T
_{2},\sigma)$ (i.e. $h(\pi(t,x))=\sigma(t,h(x)) $ for all $t\in
\mathbb T _{1} $ and $ x \in X $, and $h $ is continuous and surjective), then the
triplet $\langle (X,\mathbb T _{1},\pi ),$ $ (Y ,$ $\mathbb T _{2},$
$ \sigma ),h \rangle $ is called {\em nonautonomous dynamical system} (NDS).
\end{definition}

\begin{example} \label{ex2.4} \rm 
An important class of NDS are generated from cocycles.
Indeed, let $\langle E, \phi ,(Y ,\mathbb R,\sigma)\rangle $ be a cocycle,
$(X,\mathbb R _{+},\pi ) $ be the associated skew-product dynamical system
($X=E\times Y, \pi =(\phi,\sigma)$) and $h= pr _{2}: X \to Y$ (the natural projection mapping), then
the triplet $\langle (X,\mathbb R _{+},\pi ),$ $(Y ,\mathbb R,\sigma
),h \rangle $ is an NDS.
\end{example}

\subsection{Conditional compactness}\label{S2.2}

Lagrange stable (or called ``compact") motions have been studied comprehensively,
but it is not the case for non-Lagrange stable motions.
The following concept of conditional compactness introduced in \cite{Che_2001} is important for our
study of noncompact motions.

\begin{definition}\label{def3.2} \rm Let $(X,h,Y)$ be a fiber
space, i.e. $X$ and $Y$ be two metric spaces and $ h: X \to Y$ be a
homomorphism from $X$ onto $Y$. A set $ M \subseteq X$ is said
to be {\em conditionally precompact} if its intersection with the preimage of any
precompact subset $ Y'\subseteq Y$, i.e. the set $h^{-1}(Y')\bigcap M$, is a
precompact subset of $X$. A set $ M $ is called {\em conditionally compact} if it is
closed and conditionally precompact.
\end{definition}

\begin{remark}\rm
1. Let $K$ be a compact space, $Y$ is a noncompact metric space,  $X:=K\times Y $ and $h= pr _{2} :
X \to Y$. Then the triplet $(X,h,Y)$ is a fiber space. The space $X$
is conditionally compact, but it is not compact.

2. If $Y$ is a compact set and $M\subseteq X$ is conditionally
precompact, then $M$ is a precompact set.
\end{remark}

The following result provides a useful criterion for conditional compactness in applications.

\begin{lemma}[\cite{CC_2009}]\label{l3.1}
Let $\langle E,\phi, (Y,\mathbb R,\sigma)\rangle$ be a cocycle
and $ \langle (X,\mathbb R_{+},\pi),(Y,\mathbb R,\sigma),h\rangle $
be the NDS generated by the cocycle $\phi$ (cf. Example \ref{ex2.4}). Assume that $x_0:=(u_0,y_0)\in X=E\times Y$ and the set
$Q_{(u_0,y_0)}^{+}:=\overline{\{\phi(t,u_0,y_0): t\in \mathbb
R_{+}\}}$ is compact. Then the semi-hull $H^{+}(x_0)$ is conditionally compact.
\end{lemma}

Now we give a concrete example to illustrate the notion of conditional compactness for noncompact motions.
To this end, we need to review some basic notions.

Denote by $C(\mathbb R)$ the family  of all continuous functions
$f:\mathbb R\to \mathbb R$  equipped with the compact-open topology.
This topology can be generated by Bebutov distance (see, e.g. \cite{{Beb_1940}}, \cite[ChIV]{sib})
$$
d(f,g):=\sup\limits_{l>0}\min\{\max\limits_{|t|\le
l}|f(t)-g(t)|, 1/l\}.
$$
Denote by $(C(\mathbb R),\mathbb R,\sigma)$ the shift dynamical
system (or called Bebutov dynamical system), i.e.
$\sigma(\tau,f):=f^{\tau}$, where $f^{\tau}(t):=f(t+\tau)$ for
$t\in\mathbb R$. Note that the function $f\in C(\mathbb R)$ is st.
$L^{+}$ (respectively, st. $L$) if and only if the function $f$ is bounded and uniformly
continuous on $\mathbb R_{+}$ (respectively, on $\mathbb R$) (see, e.g. \cite[ChIV]{sib}).

\begin{example}\label{ex3.3.1} \rm
Define $h(t):=2+\cos t +\cos \sqrt{2}t$ for $t\in \mathbb R$ , then $h$ is a Bohr
almost periodic function. The function $\varphi(t):=1/h(t)$
(respectively, $\psi(t):=\sin \varphi(t)$)
for $t\in \mathbb R$ is Levitan almost periodic\footnote{For the definitions
of Bohr and Levitan almost periodic functions, see Section \ref{S4} for details.}
(\cite[ChIV]{Lev-Zhi}), but it is not Bohr almost periodic because it
is not bounded (respectively, not uniformly continuous; see
\cite[ChV, pp.212-213]{Lev_1953} or \cite{BG}) on
$\mathbb R$. Thus the function $\varphi$ (respectively, $\psi$) is
not st. $L$. Denote by
$Y=H(\varphi)=\overline{\{\varphi^{\tau}:\ \tau\in\mathbb R\}}$
(respectively,
$X=H(\psi,\varphi)=\overline{\{(\psi^{\tau},\varphi^{\tau}):\
\tau\in\mathbb R\}}$), where by bar we mean the closure in
$C(\mathbb R)$ (respectively, in $C(\mathbb R)\times C(\mathbb R)$).

\begin{lemma}\label{lC}
Let $\varphi,\psi$ and $X,Y$ be as in Example \ref{ex3.3.1}.
Consider the NDS $\langle (X,\mathbb R,\pi),$ $(Y,\mathbb
R,\sigma),$ $h \rangle$, where $h=pr_{2}:X\to Y$, and $(Y,\mathbb
R,\sigma)$ and $(X,\mathbb R,\pi)$ are the shift dynamical systems on $Y$ and $X$ respectively.
Then the following statements hold:
\begin{enumerate}
\item the set $Y$ (respectively, $X$) is not compact in $C(\mathbb
R)$ (respectively, $C(\mathbb R)\times C(\mathbb R)$);
\item the set $X$ is a conditionally compact.
\end{enumerate}
\end{lemma}

\begin{proof}
The first statement follows from the construction of
$Y$ (respectively, $X$) because the function $\varphi$
(respectively, $(\psi,\varphi)$) is not st. $L$.

To finish the proof of the lemma it is sufficient to establish that $X$
is conditionally precompact because it is closed. Let $K'$ be an
arbitrary precompact subset of $Y$ and
$K=h^{-1}(K')\subseteq X$. We need to show that $K$ is precompact.
Consider an arbitrary sequence $\{(\psi_n,\varphi_n)\}\subset K$,
then, by the definition of hull, for each $n\in \mathbb N$ there exists a number
$\tau_{n}\in\mathbb R$ such that
\begin{equation}\label{eqK}
d(\psi^{\tau_n},\psi_{n})<1/n\ \ \mbox{and}\ \
d(\varphi^{\tau_n},\varphi_{n})<1/n.
\end{equation}
Since $\{\varphi_{n}\}=h(\{(\psi_n,\varphi_n)\})\subseteq K'$ is
precompact, we can extract a subsequence $\{\varphi_{n_k}\}$
such that $\varphi_{n_k}\to \tilde\varphi \in H(\varphi)$ as $k\to
\infty$; note that we have $\varphi^{\tau_{n_k}}\to
\tilde\varphi$ by (\ref{eqK}). Taking into account that $\psi^{\tau}(t)=\sin
(\varphi^{\tau}(t))$ for $t\in\mathbb R$, we obtain
$\psi^{\tau_{n_k}}\to \tilde\psi:=\sin (\tilde\varphi)\in
H(\psi)$. That is, $(\psi^{\tau_{n_k}},\varphi^{\tau_{n_{k}}})\to
(\tilde\psi,\tilde\varphi)\in H(\psi,\varphi)=X$ as $k\to \infty$. Thus it follows from
(\ref{eqK}) that $(\psi_{n_k},\varphi_{n_k})\to (\tilde\psi,\tilde\varphi)\in X$ as $k\to \infty$.
The proof is complete.
\end{proof}
\end{example}

\subsection{Some general facts about NDS}

In this subsection we recall some general facts about NDS, see \cite{Che_2001} or \cite[Ch. IX]{Che_2015} for details.

\begin{definition}\label{def3.1} \rm
A point $y\in Y $ is called {\em positively} (respectively, {\em negatively}) {\em Poisson stable}
if there exists a sequence $t_{n} \to +\infty $
(respectively, $t_{n} \to -\infty $) such that $\sigma(t_{n},y) \to
y$ as $n\to\infty$. If $y $ is Poisson stable in both directions, it is called {\em Poisson stable}.
\end{definition}

Denote $\frak N _{y} :=\{\{t_{n}\}\subset \mathbb R: \sigma(t_{n},y) \to y \} $,
 $ \frak N ^{+\infty}_{y}:=\{\{t_{n}\}\in \frak N _{y}: t_{n} \to + \infty\}$,
$ \frak N _{y} ^{-\infty}:=\{\{t_{n}\}\in \frak N _{y}: t_{n} \to
- \infty \}$, and $\mathfrak N_{y}^{\infty}:=\{\{t_n\}\in \mathfrak N_{y}: t_n\to \infty\}$.

Let $ \langle (X,\mathbb R _{+},\pi),(Y,\mathbb R,\sigma),h\rangle $
be an NDS and $y \in Y$ be positively Poisson stable. Denote
\begin{equation}\label{Eplus}
\mathcal E_{y}^{+}:=\{\xi : ~\exists \{t_{n}\}\in \frak N _{y}^{+\infty}
\hbox{ such that } \pi(t_{n},\cdot)|_{X_{y}}\to \xi\},
\end{equation}
where $X_{y}:=h^{-1}(y)=\{x\in X: h(x)=y \}$ and $\to$ means the pointwise
convergence. If the NDS is two-sided, $y$ is negatively Poisson stable or Poisson stable and
we replace $\frak N _{y}^{+\infty}$ by $\frak N _{y}^{-\infty}$ or
$\frak N _{y}$ in \eqref{Eplus}, then we get the definition of
$\mathcal E_y^-$ or $\mathcal E_y$.

Let $X^{X}$ denote the Cartesian product of $X$ copies of the space
$X$, equipped with product topology. The set $X^{X}$ can be
endowed with a semigroup structure with respect to composition of
mappings from $X^{X}$ (for more details, see e.g. \cite[ChI]{bro75} and \cite{Ell_1969}).

\begin{lemma}\label{l3.4}
Let $y \in Y$ be positively Poisson stable,
$\langle (X,\mathbb R _{+},\pi ),$ $(Y,\mathbb R,\sigma),$ $h\rangle$
be an NDS and $X$ be conditionally compact. Then $\mathcal E^{+}_{y} $ is a nonempty compact
subsemigroup of the semigroup $X_{y}^{X_{y}}$.
\end{lemma}

\begin{coro}\label{cor3.5}
Let $y \in Y$ be negatively Poisson stable, $\langle (X,\mathbb R,\pi),$
$(Y,\mathbb R,\sigma),$ $h \rangle $ be a two-sided NDS and $X$ be conditionally compact, then
$\mathcal E^{-}_{y}$ is a nonempty compact subsemigroup of the semigroup
$X_{y }^{X_{y}}$.
\end{coro}

\begin{lemma}\label{l3.6} Let $y \in Y$ be Poisson stable, $\langle (X,\mathbb R,\pi),$ $(Y,\mathbb
R,\sigma ),h \rangle $ be a two-sided NDS
and $X$ be conditionally compact, then $\mathcal E_{y}$ is a
nonempty compact subsemigroup of the semigroup $X_{y}^{X_{y}}$.
\end{lemma}

\begin{coro} \label{cor3.7} Under the conditions of Lemma \ref{l3.6}
$\mathcal E^{+}_{y}$ and $\mathcal E^{-}_{y}$ are two nonempty
subsemigroups of the semigroup $\mathcal E_{y}$.
\end{coro}

\begin{lemma}\label{l3.8} Under the conditions of Lemma \ref{l3.6}
the following statements hold:
\begin{enumerate}
\item if $\xi _{1} \in \mathcal E^{-}_{y}$ and $\xi _{2} \in
\mathcal E^{+}_{y},$ then $\xi _{1} \cdot \xi _{2} \in \mathcal
E^{-}_{y}\bigcap \mathcal E^{+}_{y},$ where $\xi_1\cdot \xi_2$ is
the composition of $\xi_1$ and $\xi_2$;

\item $\mathcal E^{-}_{y}\bigcap \mathcal E^{+}_{y}$ is a
subsemigroup of the semigroups $\mathcal E^{-}_{y}, \mathcal
E^{+}_{y}$ and $\mathcal E_{y}$;

\item $\mathcal E^{-}_{y}\cdot \mathcal E_{y}\subseteq \mathcal
E^{-}_{y}$ and
 $\mathcal E^{+}_{y}\cdot \mathcal E_{y}\subseteq \mathcal E^{+}_{y}$, where
$ A_{1}\cdot A_{2}:=\{\xi _{1} \cdot \xi _{2}: \xi _{i} \in
A_{i}, i=1,2\}$ and $ A_{i} \subseteq \mathcal  E_{y} $;

\item  if at least one of the subsemigroups $\mathcal E^{-}_{y}$
or $\mathcal E^{+}_{y}$ is a group, then $\mathcal E^{-}_{y}=
\mathcal E^{+}_{y} =\mathcal E_{y}$.
\end{enumerate}
\end{lemma}

\begin{lemma}\label{l3.9}
Let $y \in Y $ be Poisson
stable, $\langle (X,\mathbb R,\pi),$ $(Y,\mathbb
R,\sigma),h\rangle $ be a two-sided NDS,
$X$ be conditionally compact and
\begin{equation}\label{eq3.1}
\inf \limits _{n \in \mathbb N}\rho (\pi(t_{n}, x_{1}),\pi(t_{n}, x_{2}))>0
\end{equation}
for all $\{t_{n}\}\in \frak N ^{-\infty}_{y}$ and $x_{1},x_{2} \in
X_{y}~ (x_{1}\not= x_{2}),$ then $\mathcal E^{-}_{y}$ is a
subgroup of the semigroup $\mathcal E_{y}$.
\end{lemma}

\begin{coro}\label{cor3.9.1}
Under the conditions of Lemma
\ref{l3.9} the following statements hold:
\begin{enumerate}
\item $\mathcal E^{-}_{y}=\mathcal E^{+}_{y}=\mathcal E_{y}$;
\item $\mathcal E^{-}_{y}$ (respectively, $\mathcal E^{+}_{y}$ and
$\mathcal E_{y}$) is a group.
\end{enumerate}
\end{coro}

\begin{lemma}\label{l3.10}
Assume that the conditions of Lemma \ref{l3.9} hold, then inequality (\ref{eq3.1}) is also
fulfilled for any $\{t_{n}\}\in \frak{N} ^{+\infty}_{y}$ and
$x_{1},x_{2} \in X_{y}\ (x_{1}\not= x_{2})$.
\end{lemma}

\begin{definition}\label{def3.11} \rm Let
$\langle E,\phi, (Y,\mathbb R,\sigma)\rangle$  (respectively,
$(X,\mathbb R_{+},\pi)$) be a cocycle (respectively, one-sided
dynamical system). A continuous mapping $\nu :\mathbb R \to E $
(respectively, $\gamma :\mathbb R \to X$) is called an {\em entire
trajectory} of cocycle $\phi $ (respectively, of dynamical system
$(X,\mathbb R_{+},\pi)$) passing through the point $(u,y)\in E\times
Y$ (respectively, $x\in X$) at $t=0$ if $ \phi (t,\nu
(s),\sigma(s,y))=\nu (t+s) \ \mbox{and} \ \nu (0)=u $ (respectively,
$\pi (t,\gamma (s))=\gamma (t+s)\ \mbox{and} \ \gamma(0)=x$) for all
$ t\in \mathbb R_{+} $ and $s\in \mathbb R.$
\end{definition}

Denote by
\begin{enumerate}
\item[-]
$C(\mathbb R,X)$ the space of all continuous functions $f:\mathbb
R\to X$ equipped with the compact-open topology;

\item[-] $\Phi_{x}$ the family of all entire trajectories of
$(X,\mathbb R_{+},\pi)$ passing through the point $x\in X$ at the
initial moment $t=0$ and $\Phi:=\bigcup \{\Phi_{x}:\ x\in X\}$.
\end{enumerate}

\begin{remark}\label{remIC} \rm
Note that:
\begin{enumerate}
\item the compact-open topology on the space $C(\mathbb R,X)$ is
metrizable, for example by Bebutov distance
\begin{equation}\label{eqD1}
d(\varphi,\psi):=\sup\limits_{l>0}d_{l}(\varphi,\psi),\nonumber
\end{equation}
where $d_{l}(\varphi,\psi):=\min\{\max\limits_{|t|\le
l}\rho(\varphi(t),\psi(t)), 1/l\}$;

\item if $\gamma \in \Phi_{x}$ then $\gamma^{\tau}\in
\Phi_{\gamma(\tau)}$, where $\gamma^{\tau}(t):=\gamma(t+\tau)$ for
$t\in\mathbb R$, and consequently $\Phi$ is a translation
invariant subset of $C(\mathbb R,X)$;

\item if $\gamma_{n}\in \Phi_{x_n}$ and $\gamma_{n}\to \gamma$ in
$C(\mathbb R,X)$ as $n\to \infty$, then $\gamma \in \Phi_{x}$ with
$x:=\lim\limits_{n\to \infty}x_{n}$ and consequently $\Phi$ is a
closed subset of $C(\mathbb R,X)$.
\end{enumerate}
\end{remark}

Similar to the shift dynamical system $(C(\mathbb R),\mathbb R,\sigma)$ in Section \ref{S2.2},
let $\big{(}C(\mathbb R,X),\mathbb R,\lambda \big{)}$ be the shift
dynamical system (or Bebutov dynamical system, see e.g. \cite{bro75,Che_2015,Sel_71,scher72}) on the
space $C(\mathbb R,X)$. By Remark \ref{remIC} $\Phi$ is a closed and
invariant (with respect to shifts) subset of $C(\mathbb R,X)$, and
consequently on $\Phi$ is defined a shift dynamical system
$(\Phi,\mathbb R,\lambda)$ induced from $\big{(}C(\mathbb R,X),\mathbb
R,\lambda \big{)}$.

\section{Structure of the $\omega$-limit set}\label{S3}

Let $M$ be a subset of $X$. We denote the $\omega$-limit set of $M$ by
\[
\omega
(M):= \bigcap\limits_{t\ge 0}\overline{\bigcup\{\pi(\tau,M):\ \tau
\ge t\}};
\]
for a singleton set, for simplicity we also write $\omega(x)$ or $\omega_x$ for $\omega(\{x\})$
and denote $\omega_q(M):=\omega(M)\bigcap h^{-1}(q)$.
Note that $x\in\omega(M)$ if and only if there exists
sequences $\{x_n\}\subset M$ and $\{t_n\}\subset\mathbb R$ such that $t_n\to +\infty$ as $n\to \infty$ and
$\lim\limits_{n\to \infty}\pi(t_n,x_n)=x$.

\begin{lemma}\label{l3.12}
Let $\langle (X,\mathbb R ,\pi),$ $(Y,\mathbb R,\sigma)$ $,h\rangle $ be
an NDS, $M\subseteq X $ be a nonempty subset, and $\Sigma_{M}^{+}=\{\pi(t,M):
t\ge 0\}$ be conditionally precompact. Then for any $x\in \omega(M)$
there exists at least one entire trajectory $\gamma$ of dynamical system
$(X,\mathbb R_{+} ,\pi )$ passing through the point $x$ at $t=0$
and $ \gamma (\mathbb R)\subseteq \omega(M)\ ( \gamma (\mathbb
R):=\{\gamma (t): t\in \mathbb R \})$.
\end{lemma}


\begin{proof}
Let $ x\in \omega (M)$, then there are $\{t_n\}\subset \mathbb R $
and $\{x_n\} \subset M$ such that $x=\lim \limits_{n \to
\infty}\pi(t_n,x_n)$ and $t_n \to +\infty $ as $n\to \infty$. We
consider the sequence $ \{\gamma_n \}\subset C(\mathbb R,X)$ defined
by
$$ \gamma_n (t):=\pi(t+t_n,x_n)\quad \mbox{for}\quad t\ge -t_n \quad
\mbox{and} \quad \gamma _n (t):=x_n \quad \mbox{for}\quad  t\le
-t_n.$$

We now show that the above sequence $\{\gamma _n\}$ is equicontinuous on any compact interval.
If this is not true, then there exist
$\varepsilon _0, l_0>0, t^i_n \in [-l_0,l_0]$ and $\delta _n \to 0 $ such that
\begin{equation}\label{eq9.2.3}
\vert t^{1}_n-t^{2}_n\vert \le \delta _n \quad \mbox{and} \quad \rho
(\gamma _n (t^{1}_n ), \gamma _n (t^{2}_n ))\ge \varepsilon _0.
\end{equation}
We may suppose that $t_n^{i}\to t_0 ~~ (i=1,2).$ Since $t_n\to
+\infty$ as $n\to \infty$, there exists a number $n_0\in
\mathbb N$ such that $t_n\ge l_0$ for any $n\ge n_0$. From
(\ref{eq9.2.3}) we obtain
\begin{equation}\label{eq9.2.4}
\varepsilon _0 \le \rho (\gamma _n (t_n^{1}),\gamma _n
(t_n^{2}))=\rho (
\pi(t_n^{1}+l_0,\pi(t_n-l_0,x_n)),\pi(t_n^{2}+l_0,\pi(t_n-l_0,x_n)))
\end{equation}
for any $ n \ge n_0$. Note that $y_n:=h(\pi(t_n,x_n))\to h(x)=:y$ and
$h(\pi(t_n-l_0,x_n))=\sigma(-l_0,y_n)\to \sigma(-l_0,y)$ as $n\to
\infty$. Since $\Sigma_{M}^{+}$ is conditionally precompact,
the sequence $\{\pi(t_n-l_0,x_n)\}$ is relatively compact. Without
loss of generality we suppose that $\{\pi(t_n-l_0,x_n)\}$
converges and denote by $\bar{x}$ its limit. Then passing to limit in
(\ref{eq9.2.4}) we obtain
\begin{equation}\label{eqC1}
\varepsilon _0 \le \rho (
\pi(t_0+l_0,\bar{x}),\pi(t_0+l_0,\bar{x}))=0, \nonumber
\end{equation}
a contradiction.

Next we prove that the set $\{\gamma_n(t): t\in [-l,l],\ n\in
\mathbb N\}$ is precompact for any $l>0$. To this end, note that for any $n\ge
n_0$ we have
$h(\gamma_n(t))=h(\pi(t+t_n,x_n))=\sigma(t,y_n)$, where
$y_n=h(\pi(t_n,x_n))\to h(x)=y$ as $n\to \infty$. So the set
$K:=\{\sigma(t,y_n): t\in [-l,l], n\in\mathbb N\}\subset Y$ is
precompact. Since the set $\Sigma_{M}^{+}$ is conditionally
precompact and $\{\gamma_n(t): t\in [-l,l], n\in\mathbb N\}\subset
h^{-1}(K)\bigcap \Sigma_{M}^{+}$, the set $\{\gamma_n(t): t\in
[-l,l], n\in\mathbb N\}$ is precompact.

It follows from the Arzel\`a--Ascoli theorem that $\{\gamma _n\}$ is a relatively compact sequence of $C(\mathbb R,X)$.
Let $\gamma $ be a limit point of the sequence $\{\gamma _n\}$,
then there exists a subsequence $ \{\gamma _{n_k}\}$ such that $
\gamma (t)=\lim \limits _{n \to \infty} \gamma _{n_k} (t) $
uniformly on every compact interval. In particular,
$\gamma(t)\in \omega(M)$ for any
$t\in \mathbb R$ because $\gamma(t)=\lim\limits_{n\to
\infty}\pi(t+t_{n_k},x_{n_k})$. We note that
$$
\pi ^t\gamma(s)=\lim \limits
_{n \to \infty}\pi^t\gamma _{n_k}(s)= \lim \limits _{n \to
\infty}\gamma _{n_k}(s+t)=\gamma (s+t)
$$
for all $t\in \mathbb R_{+}$ and $s\in \mathbb R$. Finally, we see
that $\gamma (0)=\lim \limits _{n \to \infty}\gamma _{n_k}(0)= \lim
\limits _{n\to \infty}\pi(t_{n_k}, x_{n_k})=x$, i.e. $\gamma$ is an
entire trajectory of dynamical system $(X,\mathbb R _+,\pi)$ passing
through point $x$. The proof is complete.
\end{proof}

\begin{remark}\label{remC}\rm 1. If $\Sigma^{+}_{h(M)}$ is precompact,
then so is $\Sigma^{+}_{M}$.

2. If $\Sigma^{+}_{h(M)}$ is not precompact, then $\Sigma^{+}_{M}$,
generally speaking, is not precompact.
\end{remark}

\begin{theorem}\label{th01}
Let  $\langle (X,\mathbb R_{+},\pi),(Y,\mathbb R,\sigma),h\rangle$
be an NDS, $x_0\in X$, $\Sigma_{x_0}^{+}$ be conditionally
precompact and $\omega_{y_0}\not= \emptyset$, where
$y_0:=h(x_0)$. Then the following statements hold:
\begin{enumerate}
\item $\omega_{x_0}\bigcap X_{q}\not= \emptyset$ for any $q\in
\omega_{y_0}$ (recall that $X_q=h^{-1}(q)$), and consequently $\omega_{x_0}\not= \emptyset$;

\item $h(\omega_{x_0})=\omega_{y_0}$;

\item the set $\omega_{x_0}$ is conditionally compact;

\item $\pi(t,\omega_q({x_0}))=\omega_{\sigma(t,q)}({x_0})$ for any
$t\in\mathbb R_{+}$ and $q\in \omega_{y_0}$, recalling that
$\omega_q({x_0})=\omega({x_0})\bigcap X_{q}$;

\item $\omega_{x_0}$ is invariant, i.e. $\pi(t,\omega_{x_0})=\omega_{x_0}$ for any $t\ge
0$.
\end{enumerate}
\end{theorem}

\begin{proof}
(i) Let $q\in \omega_{y_0}$, then there exists a sequence
$\{\tau_n\}\subset \mathbb R_+$ such that $\tau_n\to +\infty$ and
$\sigma(\tau_n,y_0)\to q$ as $n\to \infty$. Denote
$K:=\{\sigma(\tau_n,y_0): n\in\mathbb N\}$, then the set
$\{\pi(\tau_n,x_0): {n\in\mathbb N}\}=\Sigma_{x_0}^{+}\bigcap
h^{-1}(K)$ is precompact. Without loss of generality we suppose
that the sequence $\{\pi(\tau_n,x_0)\}$ is
convergent, and denote $p:=\lim\limits_{n\to \infty}\pi(\tau_n,x_0)$.
Thus $p\in \omega_{x_0}\bigcap X_{q}\not= \emptyset$.

(ii) For given $p\in \omega_{x_0}$,  there is a sequence $t_n\to +\infty$ as
$n\to \infty$ such that $\lim\limits_{n\to \infty}\pi(t_n,x_0)=p$,
and consequently $q:=h(p)=\lim\limits_{n\to
\infty}\sigma(t_n,y_0)\in \omega_{y_0}$. Thus we have
$h(\omega_{x_0})\subseteq \omega_{y_0}$. Let now $q\in \omega_{y_0}$
and $\tau_{n}\to +\infty$ as $n\to \infty$ such that
$q=\lim\limits_{n\to \infty}\sigma(\tau_n,y_0)$. Denote by
$K:=\{\sigma(\tau_n,y_0): n\in\mathbb N\}$. Since
$\Sigma_{x_0}^{+}$ is conditionally precompact and
$\{\pi(\tau_n,x_0)\}\subseteq \Sigma_{x_0}^{+}\bigcap h^{-1}(K)$,
the set $\{\pi(\tau_n,x_0)\}$ is precompact. Without loss of generality
we suppose that the sequence $\{\pi(\tau_n,x_0)\}$ converges and
denote by $p:=\lim\limits_{n\to \infty}\pi(\tau_n,x_0)$. It is clear
that $q=h(p)\in h(\omega_{x_0})$.

(iii) Let $K\subseteq Y$ be a precompact subset, $M:=\omega_{x_0}\bigcap
h^{-1}(K)$, $\{x_n\}\subseteq M$ be a sequence and $y_n:=h(x_n)\in
K\bigcap \omega_{y_0}$ for $n\in\mathbb N$. Then there exists a
sequence $\{\tau_n\}\subset \mathbb R_{+}$ such that $\tau_n>n$ and
\begin{equation}\label{eqO1}
\rho(y_n,\sigma(\tau_n,y_0))<1/n
\end{equation}
for $n\in\mathbb N$. Since $y_n=h(x_n)\in K$, we may
suppose that the sequence $\{y_n\}$ converges and denote its limit by
$q$. Then we have $q=\lim\limits_{n\to \infty}\sigma(\tau_n,y_0)$ by (\ref{eqO1}). On the other hand
$\{\pi(\tau_n,x_0)\}\subseteq \Sigma^{+}_{x_0}\bigcap h^{-1}(K')$
($K':=\{\sigma(\tau_n,y_0)\}$), and consequently
the set $\{\pi(\tau_n,x_0): n\in\mathbb N\}$ is precompact. Taking into consideration
(\ref{eqO1}) we obtain that the sequence $\{x_n\}$ is also
precompact. So the set $\omega_{x_0}$ is conditionally
precompact; furthermore, it is conditionally compact because it is closed.

(iv) 
Note first that $\pi(t,\omega_{x_0}\bigcap
X_{q})\subseteq \pi(t,\omega_{x_0})\bigcap \pi(t,X_{q})\subseteq
\omega_{x_0}\bigcap X_{\sigma(t,q)}= \omega_{\sigma(t,q)}({x_0})$ for
any $t\ge 0$ and $q\in\omega_{y_0}$. Let now $t\ge 0$, $q\in
\omega_{y_0}$ and $x\in \omega_{\sigma(t,q)}({x_0})$, then by Lemma
\ref{l3.12} there exists at least one entire trajectory $\gamma$
passing through the point $x$ at the moment $t=0$ such that
$h(\gamma(\tau))=\sigma(\tau,h(x))=\sigma(\tau,\sigma(t,q))$ for any
$\tau\in \mathbb R$. In particular, we have
$h(\gamma(-t))=\sigma(-t,\sigma(t,q))=q$, and consequently
$\gamma(-t)\in \omega_q({x_0})$. Therefore, we have
$x=\pi(t,\gamma(-t)) \in \pi(t,\omega_q({x_0}))$.

(v) It is sufficient to show $\omega_{x_0}\subseteq \pi(t,\omega_{x_0})$ because the inverse
inclusion is evident. Let $x\in \omega_{x_0}$, then by Lemma
\ref{l3.12} there exists at least one entire trajectory $\gamma$ lying on $\omega_{x_0}$ and
passing through the point $x$ at the moment $t=0$. Since
$\gamma(-t)\in \omega_{x_0}$ and $x=\gamma
(t-t)=\pi(t,\gamma(-t))\in \pi(t,\omega_{x_0})$, we have $\omega
_{x_0}\subseteq \pi(t,\omega_{x_0})$.
The proof is now complete.
\end{proof}

\begin{remark}\label{remC1} \rm
If $y_0$ is positively Lagrange stable, then the
results of Theorem \ref{th01} are known. The novelty
here is that $y_0$ is not necessarily Lagrange stable in the
positive direction, and hence our results apply to noncompact
Poisson stable motions.
\end{remark}

\begin{definition}\rm
Let $\langle (X,\mathbb R_{+},\pi), (Y,\mathbb R,\sigma),h\rangle$
be an NDS. A subset $A\subseteq X$ is said to be {\em (positively) uniformly
stable} if for arbitrary $\varepsilon
>0$ there exists $\delta =\delta(\varepsilon)>0$ such that
$\rho(x,a)<\delta$ ($a\in A,\ x\in X$ and $h(a)=h(x)$) implies
$\rho(\pi(t,x),\pi(t,a))<\varepsilon$ for any $t\ge 0$.
In particular, a point $x_0\in X$ is
called uniformly stable if the singleton set
$\{x_0\}$ is so.
\end{definition}

\begin{remark}\rm Let $A\subseteq X$ be uniformly stable and $B\subseteq A$, then $B$ is also uniformly
stable.

\end{remark}

\begin{lemma}\label{lBC1}(\cite[ChIV]{bro75}, \cite{BC_1974})
If the set $A\subseteq X$ is uniformly stable and the mapping $h:X\to Y$ is open, then the closure
$\overline{A}$ of $A$ is uniformly stable.
\end{lemma}


\begin{coro}\label{lUS1} If $\Sigma_{x_0}^{+}$ is uniformly stable and $h$ is open, then:
\begin{enumerate}
\item
 $H^{+}(x_0)$ is uniformly stable;
\item
 $\omega_{x_0}$ is uniformly stable, because $\omega_{x_0}\subseteq H^{+}(x_0)$.
\end{enumerate}
\end{coro}

\begin{remark}\label{open}\rm
Note that if an NDS is generated by a skew-product dynamical system (or equivalently by a cocycle)
in which case the homomorphism $h$ is given by
the natural projection mapping, then clearly $h$ is open.
\end{remark}

\begin{theorem}\label{th02}
Let $\langle (X,\mathbb R_{+},\pi ),
(Y,\mathbb R,\sigma ), h\rangle $ be an NDS with the following properties:
\begin{enumerate}
\item there exists a point $x_0\in X$ such that the positive
semi-trajectory $\Sigma_{x_0}^{+}$ is conditionally precompact;
\item the set $\omega_{x_0}$ is positively uniformly stable.
\end{enumerate}
Then all motions on $\omega_{x_0}$ can be extended uniquely to the
left and on $\omega_{x_0}$ is defined a two-sided dynamical system
$(\omega_{x_0},\mathbb R, \pi)$, i.e. the one-sided dynamical
system $(X,\mathbb R_+,\pi)$ generates on $\omega_{x_0}$ a two-sided
dynamical system $(\omega_{x_0},\mathbb R,\pi )$.
\end{theorem}

\begin{proof}
We divide the proof into three steps.

Step 1: we prove that the set $\omega_{x_0}\subset X$ is
distal in the negative direction with respect to the NDS $\langle (X,\mathbb R_{+},\pi),(Y ,\mathbb
R,\sigma),h\rangle $, i.e.
\begin{equation}\label{eq9.3.1}
\inf \limits _{t\le 0}\rho (\gamma _1 (t),\gamma _2 (t))>0
\end {equation}
for all $\gamma _i \in \tilde{\Phi} _{x_i} (i=1,2$ and $x_1\not=
x_2),$ where $\tilde{\Phi} _{x} $ denotes the family of all
entire trajectories of $(X,\mathbb R _{+},\pi)$ passing through
point $x$ and belonging to $\omega_{x_0}$. If this is not true, then
there exist $ y_0 \in Y , x^0_i \in\omega_{x_0}\bigcap X_{y_0} ~
(h(x_1^{0})=h(x_2^{0}),\ x^0_1\not=x^0_2), \gamma ^0 _i \in
\tilde{\Phi}_{x^0_i} \ (i=1,2)$ and $t_n \to +\infty$ such that
\begin{equation}\label{eq9.3.2}
\rho (\gamma _1^0(-t_n),\gamma _2^0(-t_n)) \to 0
\end{equation}
as $n \to \infty $. Let $ \varepsilon_{0} :=\rho (x^0_1,x^0_2)>0 $
and $\delta_{0} =\delta (\varepsilon_{0}) >0$ be chosen from the
positively uniform stability of $\omega_{x_0}$. Then we have
$$
\rho (\gamma _1^0(-t_n),\gamma _2^0(-t_n))<\delta_{0}
$$
for sufficiently large $n$ by (\ref{eq9.3.2}), so $ \varepsilon_{0} =\rho (x^0_1,x^0_2)=\rho
(\pi(t_n, \gamma _1^0(-t_n)),\pi(t_n, \gamma
_2^0(-t_n)))<\varepsilon_{0}$, a contradiction.

Step 2: we will show that for any $x \in \omega_{x_0}$ the set
$\tilde{\Phi}_{x}$ is a singleton set. Let
$\tilde{\Phi} :=\bigcup \{\tilde{\Phi}_{x}\ |\ x\in \omega_{x_0} \}
\subset C(\mathbb R, X)$. It is immediate to check that $\tilde{\Phi} $
is a closed invariant subset of dynamical system $(C(\mathbb
R,X),\mathbb R, \lambda)$, so on the set $\tilde{\Phi} $ is induced a dynamical system $(\tilde{\Phi} ,\mathbb
R,\lambda)$. Let $ H$ be a mapping from $ \tilde{\Phi} $ onto
$\omega_{y_0}$, defined by $ H(\gamma):=h(\gamma (0))$.
Then it can be shown (see, e.g. \cite[ChII]{Che_2010}) that the
triplet $\langle (\tilde{\Phi}, \mathbb R,\lambda), (\omega_{y_0},
\mathbb R,\sigma ),H \rangle $ is an NDS.
This NDS is distal in the negative direction, i.e.
$$
\inf \limits_{t\le 0}  d (\gamma_{1}^{t},\gamma_{2}^{t})>0
$$
for all $\gamma_{1},\gamma_{2} \in H^{-1}(y) ~ (\gamma_{1}\not=
\gamma_{2})$ and $y \in \omega_{y_0}$, recalling that
$\gamma^{\tau}:=\sigma (\tau, \gamma)$, i.e. $\gamma
^{\tau}(s)=\gamma (\tau +s) $ for $ s\in \mathbb R$. Indeed, otherwise there
exist $ \bar{y} , \bar{\gamma} _1, \bar{\gamma} _2 \in
H^{-1}(\bar{y})\ ( \bar{\gamma}_1 \not= \bar{\gamma}_2 ) $ and $ t_n
\to + \infty $ such that $  d
(\bar{\gamma}_1^{-t_n},\bar{\gamma}_2^{-t_n})\to 0 $ as $n \to
\infty $ and consequently
\begin{equation}\label{eq9.3.3}
\rho(\bar{\gamma}_1(-t_n), \bar{\gamma}_2(-t_n)) \le d
(\bar{\gamma}_1^{-t_n},\bar{\gamma}_2^{-t_n})\to 0.
\end{equation}
Since $\bar{\gamma}_1 \not= \bar{\gamma}_2 $, there exists $
t_0 \in \mathbb R$ such that $ \bar{\gamma}_1
(t_0)\not=\bar{\gamma}_2 (t_0)$. Let $ \tilde {\gamma}_i(t)
:=\bar{\gamma} _i(t+t_0)$ for $t\in \mathbb R $, then $ \tilde
{\gamma} _i \in \tilde{\Phi}_{\sigma(t_0,\bar{y})}$ and by (\ref{eq9.3.3}) we have
\begin{equation}\label{eq9.3.4}
\rho (\tilde {\gamma}_1(-t_n), \tilde{\gamma}_2(-t_n)) \to 0 \quad\hbox{  as } n \to \infty.
\end{equation}
Thus we have found $ q:=h(\bar{\gamma}
_i (t_0))$, $x_i:=\bar{\gamma}_i (t_0) ~ (i=1,2),\
h(x_1)=h(x_2)~ (x_1\not=x_2)$ and the entire trajectories $
\tilde{\gamma _i}\in \tilde{\Phi}_{x_i}$ such that $
\tilde{\gamma _1}$ and $ \tilde{\gamma _2}$ are proximal (see
(\ref{eq9.3.4})). But (\ref{eq9.3.4}) and (\ref{eq9.3.1}) are
contradictory, so the negative distality of the NDS $\langle (\tilde{\Phi}, \mathbb R,\sigma),
(\omega_{y_0}, \mathbb R,\sigma ),H \rangle $ is proved.

If there exist $p\in \omega_{x_0} $ and two
different trajectories $\gamma _1 ,\gamma _2 \in \tilde{\Phi}_{p}$, then in virtue of the distality of
$\gamma _1$ and $\gamma _2$ we have
$$
\alpha (\gamma _1,\gamma _2):=\inf \limits _{t\le 0} d (\gamma
_1^{t},\gamma _2^{t})> 0.
$$
So $ \rho (\gamma _1(t),\gamma _2(t))\ge \alpha
(\gamma _1,\gamma _2)> 0 $ for all $t\le 0$. In particular $\gamma
_1 (0)\not= \gamma _2 (0)$, a contradiction.

Step 3: let now $ \tilde {\pi}: \mathbb R
\times \omega_{x_0} \to \omega_{x_0}$ be a mapping defined by
$$
\tilde {\pi}(t,x):=\pi (t,x)\quad \mbox{if} \quad  t\ge 0 \quad
\mbox{and} \quad  \tilde {\pi}(t,x):=\gamma _x(t) \quad \mbox{if}
\quad t<0
$$
for $x \in \omega_{x_0}$, then $(\omega_{x_0},\mathbb R,\tilde{\pi})$ is a two-sided dynamical
system. Here $\gamma _x $ is the unique entire
trajectory of the dynamical system $(X,\mathbb R_{+},\pi)$ passing
through point $x$ and belonging to $\omega_{x_0}$. To prove that $
(\omega_{x_0},\mathbb R,\tilde{\pi})$ is a two-sided dynamical
system it suffices to check the continuity of
the mapping $ \tilde{\pi} $. Let $x\in \omega_{x_0},\ t \in \mathbb
R_{-},\ x_n \to x$ and $t_n \to t$. Then there is an $l_0>0$ such
that $t_n\in [-l_0,l_0]$ for $n\in\mathbb N$ and consequently
\begin{align}\label{eq9.3.5}
\rho (\tilde {\pi }(t_n,x_n),\tilde {\pi }(t,x)) & =\rho (\pi(t_n+l_0,
\gamma _{x_n}(-l_0)), \pi(t+l_0,\gamma _{x}(-l_0))) \\
& \le
\rho (\pi(t_n+l_0,\gamma _{x_n}(-l_0)),
\pi(t_n+l_0,\gamma _{x}(-l_0)))\nonumber\\
&\qquad +\rho (\pi(t_n+l_0,\gamma
_{x}(-l_0)), \pi(t+l_0,\gamma _{x}(-l_0))). \nonumber
\end{align}

We now show that the sequence $\{\gamma _{x_n}\}$ is relatively
compact in $C(\mathbb R,\omega_{x_0})$, which amounts to checking
that for arbitrary positive number $l$ the set
$M:=\{\gamma_{x_n}(t):\ t\in [-l,l],\ n\in\mathbb N\}$ is precompact
and $\{\gamma_{x_n}\}$ is equi-continuous on $[-l,l]$. Let $y_n:=h(x_n)$.
Since $y_n\to y:=h(x)$ as $n\to \infty$, the set $K:=\{\sigma(t,y_n): t\in[-l,l],n\in\mathbb N\}$ is
relatively compact. Since the set $\omega_{x_0}$ is conditionally
precompact and $h(\gamma_{x_n}(t))=\sigma(t,h(x_n))=\sigma(t,y_n)\in
K$ for $t\in [-l,l]$ and $n\in\mathbb N$, we have
$\gamma_{x_n}(t)\in \omega_{x_0}\bigcap h^{-1}(K)$. Thus the set $M$ is relatively
compact. If $\{\gamma_{x_n}\}$ is not equi-continuous on some compact interval, then
there are $\varepsilon_0>0$, $l_0>0$ (without loss of generality, this $l_0$ can be taken the same as in \eqref{eq9.3.5}),
$\delta_{n}\to 0$ ($\delta_{n}>0$) and
$t_{n}^{i}\in [-l_0,l_0]$ ($i=1,2$) such that
\begin{equation}\label{eqEQ1}
|t_n^{1}-t_{n}^{2}|<\delta_{n}\ \ \mbox{and}\ \
\rho(\gamma_{x_n}(t_n^{1}),\gamma_{x_n}(t_n^{2}))\ge
\varepsilon_{0}\nonumber
\end{equation}
for $n\in \mathbb N$. Note that
\begin{equation}\label{eqEQ2}
 \varepsilon_{0}\le
 \rho(\gamma_{x_n}(t_n^{1}),\gamma_{x_n}(t_n^{2}))=
 \rho(\pi(t_n^{1}+l_0,\gamma_{x_n}(-l_0)),\pi(t_n^{2}+l_0,\gamma_{x_n}(-l_0))).
\end{equation}
Since the sequences $\{\gamma_{x_n}(-l_0)\}$ and
$\{t_{n}^{i}\}\subset [-l_0,l_0]$ ($i=1,2$) are precompact,
we may suppose that they are convergent.
Denote $\bar{x}:=\lim\limits_{n\to \infty}\gamma_{x_n}(-l_0)$ and
$t_{0}:=\lim\limits_{n\to \infty}t_{n}^{i}$\ ($i=1,2$). Passing to
limit in (\ref{eqEQ2}) as $n\to \infty$ we obtain
\begin{equation}\label{eqEQ3}
\varepsilon_0\le
\rho(\pi(t_{0}+l_0,\bar{x}),\pi(t_{0}+l_0,\bar{x}))=0.\nonumber
\end{equation}
The obtained contradiction proves the equi-continuity of $\{\gamma_{x_n}\}$, and hence the
relative compactness of $\{\gamma_{x_n}\}$ in $C(\mathbb R,X)$.

Note that every limit point $\gamma$ of the sequence
$\{\gamma_{x_n}\}$ belongs to $ \tilde{\Phi} $ and satisfies $ \gamma (0)=x$. On the other hand,
the set $\tilde\Phi_x$ consists of the  single point $\gamma$ by Step 2, so we have
\[
\lim_{n\to\infty} \gamma_{x_n} = \gamma \quad \hbox{ in } C(\mathbb R,X).
\]
In particular, $\gamma_{x_n}(-l_0)\to \gamma(-l_0)$ as $n\to \infty$. Taking limit in (\ref{eq9.3.5}) as $n \to \infty$ we obtain
the continuity of mapping $\tilde{\pi} $ in $(t,x)$. The
theorem is completely proved.
\end{proof}

\begin{coro}\label{corO2} Let $\langle (X,\mathbb R_{+},\pi ),
(Y,\mathbb R,\sigma ), h\rangle $ be an NDS with the following properties:
\begin{enumerate}
\item there exist two points $x_0^{i}\in X$ ($i=1,2$) such that the positive
semi-trajectories $\Sigma_{x_0^{i}}^{+}$ ($i=1,2$) are conditionally
precompact;
\item the sets $\omega_{x_0^{i}}$ ($i=1,2$) are positively uniformly stable.
\end{enumerate}
Then all motions on $K:=\omega_{x_0^{1}}\bigcup \omega_{x_0^{2}}$
can be extended uniquely to the left and on $K$ is defined a
two-sided dynamical system $(K,\mathbb R, \pi)$, i.e. the one-sided
dynamical system $(X,\mathbb R_+,\pi)$ generates on $K$ a two-sided
dynamical system $(K,\mathbb R,\pi )$.
\end{coro}

\begin{remark}\label{remC2}\rm Note that Theorem \ref{th02} is
known if $Y$ is a compact minimal set (see, e.g. \cite{NOS_2007} and references therein)
or if each point of $Y$ is Poisson stable (see \cite{Che_2008}). In our case $Y$,
generally speaking, can be non-compact and non-minimal, and there is no restriction on the element of $Y$.
\end{remark}

\section{Poisson stable motions and their comparability ny character of recurrence}\label{S4}

\subsection{Classes of Poisson stable motions}

Let $(X,\mathbb R,\pi)$ be a dynamical system. Let us recall the classes of Poisson stable motions
we study in this paper, see \cite{Sel_71,scher72,Sch85,sib} for details.

\begin{definition}\label{def-stat}\rm
A point $x \in X $ is called {\em stationary} (respectively, {\em $\tau$-periodic}) if
$\pi(t,x)=x$ (respectively, $\pi(t+\tau,x)=\pi(t,x)$) for all $t\in\mathbb R$.
\end{definition}

\begin{definition} \rm
A point $x \in X $ is called {\em quasi-periodic}
if the associated function $f(\cdot):=\pi(\cdot,x): \mathbb R\to X$ satisfy
the following conditions:
\begin{enumerate}
\item the numbers $\nu_1,\nu_2,\ldots,\nu_k$ are rationally independent;
\item there exists a continuous function $\Phi :\mathbb R^{k}\to X$ such that
$\Phi(t_1+2\pi,t_2+2\pi,\ldots,t_k+2\pi)=\Phi(t_1,t_2,\ldots,t_k)$ for all $(t_1,t_2,\ldots,t_k)\in \mathbb R^{k}$;
\item $f(t)=\Phi(\nu_1 t,\nu_2 t,\ldots,\nu_k t)$ for $t\in \mathbb R$.
\end{enumerate}
\end{definition}

\begin{definition}\label{def4.2}\rm
For given $\varepsilon>0$, a number $\tau \in \mathbb R$ is called a {\em $\varepsilon$-shift of
$x$} (respectively, {\em $\varepsilon$-almost period of $x$}), if $\rho
(\pi(\tau,x),x)<\varepsilon$ (respectively, $\rho (\pi(\tau
+t,x),\pi(t,x))<\varepsilon$ for all $t\in \mathbb R$).
\end{definition}

\begin{definition}\label{def14.3}\rm
A point $x \in X $ is called {\em almost recurrent} (respectively, {\em Bohr
almost periodic}), if for any $\varepsilon >0$ there exists a
positive number $l$ such that any segment of length $l$ contains
a $\varepsilon$-shift (respectively, $\varepsilon$-almost period) of $x$.
\end{definition}

\begin{definition}\label{def4.4}\rm
If a point $x\in X$ is almost recurrent and its trajectory $\Sigma_x$ is precompact, then
$x$ is called {\em (Birkhoff) recurrent}.
\end{definition}

\begin{definition}\rm
A point $x\in X$ is called {\em Levitan almost periodic} \cite{Lev-Zhi} (see also \cite{bro75,Che_2008,Lev_1953}), if there exists a dynamical
system $(Y,\mathbb R,\sigma)$ and a Bohr almost periodic point $y\in
Y$ such that $\mathfrak N_{y}\subseteq \mathfrak N_{x}.$
\end{definition}


\begin{definition} \rm
A point $x\in X$ is called {\em almost automorphic} if it is st. $L$ and
Levitan almost periodic.
\end{definition}

\begin{definition}\rm
A point $x\in X$ is said to be {\em uniformly Poisson stable} or {\em pseudo periodic}
in the positive (respectively, negative) direction if for arbitrary $\varepsilon >0$ and $l>0$ there exists a $\varepsilon$-almost
period $\tau >l$ (respectively, $\tau < -l$) of $x$. The point $x$ is said to be uniformly Poisson stable or pseudo periodic if it is so
in both directions.
\end{definition}

\begin{definition}[\cite{Shc_1962,Shc_1963}]\label{defP1} \rm
A point $x\in X$ is said to be {\em pseudo recurrent} if for any $\varepsilon
>0,\ p\in \Sigma_{x}$ and $t_0\in
\mathbb R$ there exists $L=L(\varepsilon,t_0)>0$ such that
\begin{equation}\label{eqLP1.0}
B(p,\varepsilon)\bigcap \pi([t_0,t_0+L],p)\not=\emptyset ,\nonumber
\end{equation}
where $B(p,\varepsilon):=\{x\in X: \rho(p,x)<\varepsilon\}$ and
$\pi([t_0,t_0+L],p):=\{\pi(t,p): t\in [t_0,t_0+L]\}$.
\end{definition}




\begin{definition}\label{defSP1}
\rm A point $x\in X$ is said to be {\em strongly Poisson stable} if $p\in \omega_{p}$ for any $p\in H(x)$.
\end{definition}

\begin{remark}\label{remP3}\rm
It is known that:
\begin{enumerate}
 \item a strongly Poisson stable point is Poisson stable, but the converse is not true in general;
\item
 all the motions introduced above (Definitions \ref{def-stat}--\ref{defP1}) are strongly Poisson stable.

\end{enumerate}
\end{remark}

\begin{definition}[\cite{CS_1977,Che_2009}]\label{defAP1} \rm
A point $x\in X$ is said to be
{\em asymptotically $\mathcal P$} if there exists a
$\mathcal P$ point $p\in X$ such that
$$
\lim\limits_{t\to
+\infty}\rho(\pi(t,x),\pi(t,p))=0.
$$
Here the property $\mathcal P$ can be stationary, $\tau$-periodic, quasi-periodic, Bohr
almost periodic, almost automorphic, Birkhoff recurrent, Levitan almost periodic,
almost recurrent, pseudo periodic, pseudo recurrent, Poisson stable.
\end{definition}

\subsection{Comparability of
motions by their character of recurrence}

\subsubsection{Shcherbakov's comparability principle of
motions by their character of recurrence}
In this subsection we present some notions and results stated and proved by B. A.
Shcherbakov \cite{scher72}--\cite{Sch85}.

Let $(X,\mathbb R,\pi)$ and $(Y,\mathbb R,\sigma)$ be two dynamical
systems.

\begin{definition}\label{defShc1}\rm
A point $x\in X$ is said to be {\em comparable with $y\in Y$ by
character of recurrence} if for any $\varepsilon
>0$ there exists a $\delta =\delta(\varepsilon)>0$ such that every
$\delta$-shift of $y$ is a $\varepsilon$-shift for $x$, i.e.
$\rho(\sigma(\tau,y),y)<\delta$ implies
$\rho(\pi(\tau,x),x)<\varepsilon$.
\end{definition}


\begin{theorem}\label{tShc1}
The following conditions are
equivalent:
\begin{enumerate}
\item the point $x$ is comparable with $y$ by character
of recurrence;
\item $\mathfrak N_{y}\subseteq \mathfrak N_{x}$;
\item $\mathfrak N_{y}^{\infty}\subseteq \mathfrak
N_{x}^{\infty}$; \item from any sequence $\{t_n\}\in\mathfrak N_{y}$
we can extract a subsequence $\{t_{n_k}\}\in \mathfrak N_{x}$;
\item from any sequence $\{t_n\}\in\mathfrak N_{y}^{\infty}$ we can
extract a subsequence $\{t_{n_k}\}\in \mathfrak N_{x}^{\infty}$.
\end{enumerate}
\end{theorem}



\begin{theorem}\label{thShc2.1}
Let $x\in X$ be comparable with $y\in Y$. If
the point $y$ is stationary (respectively, $\tau$-periodic,
Levitan almost periodic, almost recurrent, Poisson stable), then so is the
point $x$.
\end{theorem}

\begin{definition}\rm
A point $x\in X$ is called \textit{uniformly comparable with $y\in
Y$ by character of recurrence} if for any $\varepsilon >0$ there
exists a $\delta =\delta(\varepsilon)>0$ such that every
$\delta$-shift of $\sigma(t,y)$ is a $\varepsilon$-shift of
$\pi(t,x)$ for all $t\in\mathbb R$, i.e.
$\rho(\sigma(t+\tau,y),\sigma(t,y))<\delta$ implies
$\rho(\pi(t+\tau,x),x)<\varepsilon$ for all $t\in \mathbb R$ (or
equivalently: $\rho(\sigma(t_1,y),\sigma(t_2,y))<\delta$ implies
$\rho(\pi(t_1,x),\pi(t_2,x))<\varepsilon$ for all $t_1,t_2\in
\mathbb R$).
\end{definition}



Denote $\mathfrak{M}_{x}:=\{\{t_{n}\}\subset\mathbb{R}:\
\{\pi(t_{n},x)\}$ converges$ \}$, $\mathfrak{M}_{x}^{+\infty}:=\{\{t_{n}\}\in\mathfrak{M}_{x}:
t_{n}\to  +\infty \ \mbox{as}\ n\to \infty\}$ and $\mathfrak{M}_{x}^{\infty}:=\{\{t_{n}\}\in\mathfrak{M}_{x}:
t_{n}\to  \infty \ \mbox{as}\ n\to \infty\}$.


\begin{definition}[\cite{Che_1977,Che_2009}] \rm
A point $x\in X$ is said to be {\em strongly
comparable with $y\in Y$ by character of recurrence} if $\mathfrak M_{y}\subseteq \mathfrak M_{x}$.
\end{definition}


\begin{theorem}\label{thShc5}
(i) If $\mathfrak M_{y}\subseteq \mathfrak M_{x}$, then $\mathfrak
N_{y}\subseteq \mathfrak N_{x}$, i.e. strong comparability implies comparability.

(ii) Let $X$ be a complete metric space. If the point $x$ is uniformly comparable with $y$ by character of recurrence,
then $\mathfrak M_{y}\subseteq \mathfrak M_{x}$, i.e. uniform comparability implies strong comparability.
\end{theorem}

\begin{theorem}\label{thShc6}
Let $y$ be Lagrange stable. Then $\mathfrak M_{y}\subseteq \mathfrak M_{x}$ holds
if and only if the point $x$ is Lagrange stable and uniformly comparable with $y$ by character of
recurrence.
\end{theorem}

\begin{theorem}\label{tS2}
Let $X$ and $Y$ be two complete metric spaces.
Let the point $x\in X$ be uniformly comparable with $y\in Y$ by character
of recurrence. If $y$ is quasi-periodic (respectively,
Bohr almost periodic, almost automorphic, Birkhoff recurrent, Lagrange stable, pseudo periodic, pseudo recurrent), then
so is $x$.
\end{theorem}

\subsubsection{Some generalization of Shcherbakov's results}

In this subsection we present some generalization of
Shcherbakov's results concerning the (uniform) comparability of points by
character of recurrence (see \cite{CC_2013}
or \cite[ChI]{Che_2017} for more details).

Let $\mathbb T_{1}\subseteq \mathbb T_{2}$ be two sub-semigroups of
group $\mathbb R$ ($\mathbb T_{i}=\mathbb R$ or $\mathbb R_{+}$ for $i=1,2$).
Consider two dynamical systems $(X,\mathbb T_1,\pi)$ and
$(Y,\mathbb T_2,\sigma)$.

\begin{theorem}\label{tG1}
Let $y\in Y$ be positively Poisson stable. Then the following conditions are
equivalent:
\begin{enumerate}
\item[a.] $\mathfrak{M}_{y}\subseteq \mathfrak{M}_{x}$; \item[b.]
$\mathfrak{M}_{y}^{\infty}\subseteq \mathfrak{M}_{x}^{\infty}$ and
$\mathfrak{N}_{y}^{\infty}\subseteq \mathfrak{N}_{x}^{\infty}$;
\item[c.] there exists a continuous mapping $h: \omega_{y}\to
\omega_{x}$ with the properties:
\begin{enumerate}
\item[(i)]
\begin{equation}\label{eqDN1}
h(y)=x;\nonumber
\end{equation}
\item[(ii)]
\begin{equation}\label{eqDN2}
h(\sigma(t,q))=\pi(t,h(q))\nonumber
\end{equation}
for all $t\in \mathbb T_1$ and $q\in\omega_{y}$.
\end{enumerate}
\end{enumerate}
\end{theorem}

\begin{theorem}\label{tG2}
Let $y\in \omega_{y}$, then the following conditions are
equivalent:
\begin{enumerate}
\item[a.] $\mathfrak{N}_{y}^{\infty}\subseteq
\mathfrak{N}_{x}^{\infty}$; \item[b.]
$\mathfrak{N}_{y}^{+\infty}\subseteq \mathfrak{N}_{x}^{+\infty}$.
\end{enumerate}
\end{theorem}

\begin{theorem}\label{tG3} Let $y\in \omega_{y}$, then the following conditions are
equivalent:
\begin{enumerate}
\item[a.] $\mathfrak{M}_{y}^{\infty}\subseteq
\mathfrak{M}_{x}^{\infty}$ and $\mathfrak{N}_{y}^{\infty}\subseteq
\mathfrak{N}_{x}^{\infty}$; \item[b.]
$\mathfrak{M}_{y}^{+\infty}\subseteq \mathfrak{M}_{x}^{+\infty}$ and
$\mathfrak{N}_{y}^{+\infty}\subseteq \mathfrak{N}_{x}^{+\infty}$.
\end{enumerate}
\end{theorem}

\section{Monotone NDS: existence of and convergence to Poisson stable motions}\label{S5}

Assume that $E$ is an ordered space.
A subset $U$ of $E$ is called lower-bounded (respectively,
upper-bounded) if there exists an element $a\in E$ such that $a\le
U$ (respectively, $a\ge U$). Such an $a$ is said to be a lower bound (respectively,
upper bound) for $U$. A lower bound $\alpha$ is said to be the
{\em greatest lower bound} (g.l.b.) or \emph{infimum}, if any other lower
bound $a$ satisfies $a\le \alpha $. Similarly, we can define the {\em least
upper bound} (l.u.b.) or \emph{supremum}.

\begin{definition}\label{defB1}
\rm Recall \cite{hew} that a {\em bundle} is a triplet $(X,h,Y)$,
where $X,Y$ are topological spaces and  $h:X\to Y$ is a continuous surjective mapping.
The space $Y$ is called the {\em base
space}, the space $X$ is called the {\em total space}, and the map $h$ is
called the {\em projection} of bundle. For each $y\in Y$, the space
$X_{y}:=h^{-1}(y)$ is called the {\em fiber} of bundle over $y\in Y$.
\end{definition}

\begin{example}\label{exB1}
\rm Let $X:=W\times Y$. A triplet
$(X,h,Y)$, where $h:=pr_{2}$ is the projection on the second factor,
is a bundle which is called the product bundle over $Y$ with fiber
$W$.
\end{example}

A bundle $(X,h,Y)$ is said to be {\em ordered} if each fiber $X_y$ is ordered. Note that only points
on the same fiber may be order related: if $x_1\le x_2$ or $x_1< x_2$, then it implies $h(x_1)=h(x_2)$.
We assume that the order relation and the topology on $X$ are compatible in the sense that $x\le \tilde x$
if $x_n\le \tilde x_n$ for all $n$ and $x_n\to x$, $\tilde x_n\to \tilde x$ as $n\to\infty$.

\begin{definition}\label{defMSP}
\rm For given bundle $(X,h,Y)$, an NDS $\langle (X,\mathbb R_{+},\pi),
(Y,\mathbb R,\sigma),h\rangle$ defined on it is said to be {\em monotone} (respectively,
{\em strictly monotone}) if $x_1\le x_2$ (respectively, $x_1< x_2$)
implies $\pi(t,x_1)\le \pi(t,x_2)$ (respectively, $\pi(t,x_1)<
\pi(t,x_2)$) for any $t>0$.
\end{definition}

For given NDS $\langle (X,\mathbb R_{+},\pi),
(Y,\mathbb R,\sigma),h\rangle$, let $\mathcal S\subseteq X$ be a nonempty ordered subset
possessing the following properties:
\begin{enumerate}
\item $h(\mathcal S)=Y$;


\item $\mathcal S$ is positively invariant with respect to $\pi$, i.e.
$\langle (\mathcal S,\mathbb R_{+},\pi),(Y,\mathbb R,\sigma),h\rangle$ is an NDS.
\end{enumerate}

Below we will use the following assumptions:
\begin{enumerate}
\item[(C1)] For every conditionally compact subset $K$ of $\mathcal S$ and $y\in Y$ the set
$K_{y}:=h^{-1}(y)\bigcap K$ has both infimum $\alpha_{y}(K)$ and supremum
$\beta_{y}(K)$.
\item[(C2)] For every $x\in \mathcal S$, the semi-trajectory
$\Sigma^{+}_{x}$ is conditionally precompact
and its $\omega$-limit set $\omega_{x}$ is positively uniformly
stable.
\item[(C3)] The NDS
\begin{equation}\label{eqNDS1}
\langle (\mathcal S,\mathbb R_{+},\pi), (Y,\mathbb R,\sigma),h\rangle
\end{equation}
is monotone.
\item[(C4)] Under condition (C1), both $\alpha_{y}(K)$ and
$\beta_{y}(K)$ belong to $K_{y}$ for any $y\in Y$.
\end{enumerate}

\begin{remark}\label{rhatemC3}
\rm Note that condition (C4) holds if fibers of the bundle
$(\mathcal S,h,Y)$ are one-dimensional, i.e. $\mathcal S_{y}= h^{-1}(y) \bigcap\mathcal S\subseteq\mathbb
R\times \{y\}$ or $\mathcal S_{y}$ is homeomorphic to a subset of $\mathbb R \times
\{y\}$ for any $y\in Y$.
\end{remark}

\subsection{(Uniform) comparability and existence of Poisson stable motions}

Firstly, we state a simple result for two points to be asymptotic which will be frequently used below.

\begin{lemma}\label{lAPS1}
Suppose that the following conditions are
fulfilled:
\begin{enumerate}
\item the points $x,x_0\in \mathcal S$ with $h(x)=h(x_0)$ are proximal,
i.e. there is a sequence $t_n\to +\infty$ as $n\to \infty$ such
that
\begin{equation}\label{eqAPS1}
\lim\limits_{n\to \infty}\rho(\pi(t_n,x),\pi(t_n,x_0))=0;
\end{equation}
\item the set $\Sigma_{x_0}^{+}\in \mathcal S$ is
positively uniformly stable.
\end{enumerate}

Then the points $x,x_0$ are asymptotic, i.e. $\lim\limits_{t\to
+\infty}\rho(\pi(t,x),\pi(t,x_0))=0$.
\end{lemma}

\begin{proof}
Let $\Sigma_{x_0}^{+}$ be positively uniformly
stable, $\varepsilon >0$ and $\delta =\delta(x_0,\varepsilon)>0$
the positive number figuring in the definition of uniform
stability. By (\ref{eqAPS1}) there exists a number
$n_0\in\mathbb N$ such that $\rho(\pi(t_n,x),\pi(t_n,x_0))<\delta$
for any $n\ge n_0$. According to the choice of the number $\delta$
we obtain $\rho(\pi(t,x),\pi(t,x_0))<\varepsilon$ for any $t\ge
t_{n_0}$. The lemma is proved.
\end{proof}

\begin{lemma}\label{lAPS2}
Assume that $(C1)$--$(C3)$ hold. For given $x_0\in \mathcal S$, let $K:=\omega_{x_0}$
and $y_0:=h(x_0)$. Then:
\begin{enumerate}
\item if $q\in \omega_{q}\subseteq
\omega_{y_0}$, $\alpha_{q}:=\alpha_{q}(K)$,
$K^{1}:=\omega_{\alpha_{q}}$, then the set
$K_{q}^{1}:=\omega_{\alpha_{q}}\bigcap h^{-1}{(q)}$ (respectively,
$\omega_{\beta_{q}}\bigcap h^{-1}{(q)}$) consists of a single point
$\gamma_{q}$ (respectively, $\delta_{q}$), i.e.
$K_{q}^{1}=\{\gamma_{q}\}$ (respectively, $\omega_{\beta_{q}}\bigcap h^{-1}{(q)}=\{\delta_{q}\}$);

\item let $\gamma_q$ and $\delta_q$ be as in (i), then we have
\[
\gamma_q \le \alpha_q \le \beta_q\le \delta_q.
\]
\end{enumerate}
\end{lemma}

\begin{proof}
Let $q$ be a point from $\omega_{y_0}$ with $q\in \omega_{q}$. We only
consider the case of $\alpha_{q}$ because the proof for $\beta_{q}$
is similar.

(i) It follows from the definition of $\alpha_{q}$ that
\begin{equation*}\label{eqP1.0}
\alpha_{q}\le x\quad \hbox{for any } x\in K_{q}=K\bigcap h^{-1}{(q)}.
\end{equation*}
Since $\pi(t,K_{q})=K_{\sigma(t,q)}$ by Theorem \ref{th01}, we have
\begin{equation}\label{eqP2.0}
\alpha_{\sigma(t,q)}\le \pi(t,x)\ \ \mbox{for any}\ \ x\in K_{q}\
\mbox{and}\ t\ge 0.
\end{equation}

We now prove that
\begin{equation}\label{eqP3.0}
\pi(t,\alpha_q)\le   \alpha_{\sigma(t,q)}\le \pi(t,x)\ \ \mbox{for
any}\ \ x\in K_{q}\ \mbox{and}\ t\ge 0.
\end{equation}
Since $K$ is invariant we have
\begin{equation}\label{eqP4.0}
\gamma(-t)\in K\ \ \mbox{for any}\ \gamma \in
\tilde{\Phi}_{x}:=\{\gamma\in \Phi_{x}: \gamma(\mathbb R)\subseteq
K\},\ x\in K_{\sigma(t,q)}\ \mbox{and}\ t\ge 0. \nonumber
\end{equation}
Note that $h(\gamma(-t))=q$, consequently $\alpha_{q}\le
\gamma(-t)$. Since the NDS (\ref{eqNDS1})
is monotone, we obtain
\begin{equation*}\label{eqP4.1.0}
\pi(t,\alpha_{q})\le\pi(t,\gamma(-t))=\gamma(0)=x\in K_{\sigma(t,q)}.
\end{equation*}
This implies that
$\pi(t,\alpha_{q})\le \alpha_{\sigma(t,q)}$ for any $t\ge 0$ because $x\in K_{\sigma(t,q)}$ is arbitrary.


Let $x_{1}\in K^1_{q}$, then there is a sequence
$t_n\to +\infty$ such that
\begin{equation}\label{eqP8.0}
\pi(t_n,\alpha_{q})\to x_{1} \ \ \mbox{and}\ \ \sigma(t_n,q)\to q
\nonumber
\end{equation}
as $n\to +\infty$. By (\ref{eqP3.0}), we have
\begin{equation}\label{eqP9.0} \pi(t_n,\alpha_{q})\le
\alpha_{\sigma(t_n,q)}.
\end{equation}
Denote by $\tilde{K}:=K\bigcup K^{1}$. By Theorem \ref{th01} both $K$ and
$K^{1}$ are conditionally compact, and hence $\tilde{K}$ is
conditionally compact. So without loss of generality we
suppose that the sequence $\pi(t_n,\cdot)\Big{|}_{\tilde{K}_{q}}$ is
convergent and denote by $\xi$ its limit; note that $\xi\in \mathcal
E_{q}^{+}$ (with $\pi$ being restricted on $\tilde K$ in the definition of $\mathcal
E_{q}^{+}$).  By Theorem \ref{th02} $\pi$ can be extended to a
two-sided dynamical system on $\tilde K$,
and by the proof of Theorem \ref{th02} the required negative separation property \eqref{eq3.1} in
Corollary \ref{cor3.9.1} also holds. Then it follows from Corollary \ref{cor3.9.1} that
$\mathcal E^{+}_{q}$ is a group, so we have
$\xi(\tilde{K}_{q})=\tilde{K}_{q}$, $\xi(K_{q})=K_{q}$ and
$\xi(K_{q}^{1})=K_{q}^{1}$. Thus, for any point $x_{2}\in K_{q}$ and
$\xi\in \mathcal E^{+}_{q}$ there exists a (unique) point
$\tilde{x}_{2}\in K_{q}$ such that $\xi(\tilde{x}_{2})=x_{2}$. We
have
\begin{equation}\label{eqP10.0}
\sigma(t_n,q)\to q \ \ \mbox{and}\ \ x_{2}=\lim\limits_{n\to
\infty}\pi(t_n,\tilde{x}_{2}).
\end{equation}
Combining (\ref{eqP9.0}) and (\ref{eqP2.0}), we conclude that
\begin{equation}\label{eqP11.0}
\pi(t_n,\alpha_{q})\le \alpha_{\sigma(t_n,q)}\le
\pi(t_n,\tilde{x}_{2}).
\end{equation}
Letting $n\to \infty$ in (\ref{eqP11.0}), we get by \eqref{eqP10.0}
\begin{equation}\label{eqP12.0}
x_{1}\le x_{2} \nonumber
\end{equation}
for any $x_2\in K_{q}$, and hence
\begin{equation}\label{eqI1}
x\le \alpha_{q}
\end{equation}
for all $x\in K^1_{q}$.

Finally we will show under condition (\ref{eqI1}) that the set
$K^1_{q}$ consists of a single point. In fact, if
$x',x''\in K^1_{q}$ and (\ref{eqI1}) holds, then
reasoning as above we can choose a sequence $t_n\to +\infty$ as
$n\to \infty$ and $\tilde{x}''\in K^1_{q}$ such that
$\sigma(t_n,q)\to q$, $\pi(t_n,\alpha_{q})\to x'$ and
$\pi(t_n,\tilde{x}'')\to x''$. Since $\tilde{x}''\le \alpha_{q}$,
we have $\pi(t_n,\tilde{x}'')\le  \pi(t_n,\alpha_{q})$. Consequently, $x''\le x'$. Since $x',x''\in K^1_{q}$
are arbitrary, we have $x'=x''$, i.e. $K_{q}^{1}$ consists of a single
point $\gamma_{q}$.

(ii) The fact $\gamma_q\le \alpha_q$ follows from \eqref{eqI1} and the fact $K^1_{q}=\{\gamma_q\}$.

The proof is complete.
\end{proof}

\begin{coro}\label{enorbit}
Assume that the conditions of Lemma \ref{lAPS2} hold. Then $\gamma_q$
satisfies $\pi(t,\gamma_q)=\gamma_{\sigma(t,q)}$ for $t\in \mathbb R$, i.e. the mapping
$t\mapsto \gamma_{\sigma(t,q)}$ is an entire trajectory of the dynamical system
$(X,\mathbb R_+, \pi)$ passing through the point $\gamma_q$ at $t=0$. The same result holds
for $\delta_q$.
\end{coro}

\begin{proof}
It follows from Theorem \ref{th01} that $\omega_{\alpha_q}$ is an invariant set, and by
Theorem \ref{th02} the one-sided dynamical
system $(X,\mathbb R_+,\pi)$ generates on $\omega_{\alpha_q}$ a two-sided
dynamical system $(\omega_{\alpha_q},\mathbb R,\pi)$. On the other hand,
Lemma \ref{lAPS2} yields that $\omega_q(\alpha_q)=\{\gamma_q\}$, which enforces that the
required result holds.
\end{proof}

\begin{theorem}[Comparability]\label{thM1}
Assume that (C1)--(C3) hold. For given $x_0\in \mathcal S$, let $y_0:=h(x_0)$.
If $y_0\in \omega_{y_0}$, then the point $\gamma_{y_0}$
(respectively, $\delta_{y_0}$) is comparable with $y _0$ by character of
recurrence and
\begin{equation}\label{eqP0}
\lim\limits_{t\to
+\infty}\rho(\pi(t,\alpha_{y_0}),\pi(t,\gamma_{y_0}))=0
\end{equation}
\[
(\hbox{respectively, } \lim\limits_{t\to
+\infty}\rho(\pi(t,\beta_{y_0}),\pi(t,\delta_{y_0}))=0).
\]
\end{theorem}

\begin{proof}
We will only prove the result for $\gamma_{y_0}$ because the proof for $\delta_{y_0}$ is similar.

Let $\{t_n\}\in \mathfrak N^{+\infty}_{y_0}$, then
$\sigma(t_n,y_0) \to y_0$ and $t_n\to +\infty$ as $n\to \infty$.
By condition (C2) the set $\Sigma_{\gamma_{y_0}}^{+}$ is
conditionally precompact, then the sequence
$\{\pi(t_n,\gamma_{y_0})\}$ is precompact. Let $z$ be a limit
point of the sequence $\{\pi(t_n,\gamma_{y_0})\}$, then there is a
subsequence $\{t_n^{'}\}\subseteq \{t_n\}$ such that
$\pi(t^{'}_{n},\gamma_{y_0})\to z$ as $n\to \infty$. On the other
hand $\sigma(t^{'}_{n},y_0)\to y_0$ as $n\to \infty$, so $z\in \omega_{\gamma_{y_0}}\bigcap h^{-1}{(y_0)}\subseteq
\omega_{\alpha_{y_0}}\bigcap h^{-1}{(y_0)}
 =\{\gamma_{y_0}\}$ by Lemma \ref{lAPS2}, i.e. $z=\gamma_{y_0}$. Since
$\{\pi(t_n,\gamma_{y_0})\}$ is precompact and $\gamma_{y_0}$ is its
unique limit point, we have $\pi(t_n,\gamma_{y_0})\to \gamma_{y_0}$
as $n\to \infty$. That is, $\{t_n\}\in \mathfrak N^{+\infty}_{\gamma_{y_0}}$ and hence
$\mathfrak N_{y_0}^{+\infty}\subseteq \mathfrak N_{\gamma_{y_0}}^{+\infty}$.
The first statement then follows from Theorems \ref{tG2} and \ref{tShc1}.

Since $h(\alpha_{y_0})=y_0\in \omega_{y_0}$, there
exists a sequence $t_n\to +\infty$ as $n\to \infty$ such that
\begin{equation}\label{eqL1}
\sigma(t_n,y_0)\to y_0\ \ \ \mbox{as}\ \ n\to \infty \ .
\end{equation}
Taking into consideration that $\Sigma_{\alpha_{y_0}}^{+}$ is
conditionally precompact, without loss of generality we can
suppose that the sequence $\{\pi(t_n,\alpha_{y_0})\}$ converges.
Denote by $\bar{x}:=\lim\limits_{n\to \infty}\pi(t_n,\alpha_{y_0})$,
then $h(\bar{x})=y_0$ and $\bar{x}\in \omega_{\alpha_{y_0}}$, i.e. $\bar{x}\in \omega_{\alpha_{y_0}}\bigcap h^{-1}{(y_0)}$. By
Lemma \ref{lAPS2} we have $\omega_{\alpha_{y_0}}\bigcap
h^{-1}{(y_0)}=\{\gamma_{y_0}\}$, so $\bar{x}=\gamma_{y_0}$ and
consequently
\begin{equation}\label{eqL2}
\pi(t_n,\alpha_{y_0})\to \gamma_{y_0}\ \ \ \mbox{as}\ \ n\to \infty.
\end{equation}
On the other hand, by the first statement of the theorem and (\ref{eqL1}) we
obtain
\begin{equation}\label{eqL3}
\pi(t_n,\gamma_{y_0})\to \gamma_{y_0}\ \ \ \mbox{as}\ \ n\to \infty
\ .
\end{equation}
From (\ref{eqL2}) and (\ref{eqL3}) we get
\begin{equation}\label{eqL4}
\lim\limits_{n\to
\infty}\rho(\pi(t_n,\alpha_{y_0}),\pi(t_n,\gamma_{y_0}))=0\nonumber
.
\end{equation}
Now to finish the proof of equality (\ref{eqP0}) it is sufficient to
apply Lemma \ref{lAPS1}. The proof is complete.
\end{proof}

By Theorems \ref{thM1} and \ref{thShc2.1}, we have the following result:

\begin{coro}\label{corP1} Under the conditions (C1)--(C3) if the point $y_0$ is stationary
(respectively, $\tau$-periodic, Levitan almost periodic, almost recurrent, Poisson stable), then:
\begin{enumerate}
\item the point $\gamma_{y_0}$ has the same recurrent property as $y_0$;
\item the point $\alpha_{y_0}$ is asymptotically stationary
(respectively, asymptotically $\tau$-periodic, asymptotically Levitan almost periodic,
asymptotically almost recurrent, asymptotically Poisson stable).
\end{enumerate}
\end{coro}

To get the existence of more classes of Poisson
stable motions, we need to establish uniform comparability (cf. Theorem \ref{tS2}) and this reduces to verifying
strong comparability when the base space is compact (cf. Theorem \ref{thShc6}). This is what we are
doing in the following

\begin{theorem}[Strong comparability] \label{thM2}
Assume that (C1)--(C3) hold, $x_0\in \mathcal S$ and
$y_0:=h(x_0)\in Y$ is strongly Poisson stable.
Then the point $\gamma_{y_0}$ (respectively, $\delta_{y_0}$) is strongly
comparable with $y _0$ by character of recurrence and
\begin{equation}\label{eqPM3}
\lim\limits_{t\to
+\infty}\rho(\pi(t,\alpha_{y_0}),\pi(t,\gamma_{y_0}))=0
\end{equation}
\[
(\hbox{respectively, } \lim\limits_{t\to
+\infty}\rho(\pi(t,\beta_{y_0}),\pi(t,\delta_{y_0}))=0).
\]
\end{theorem}

\begin{proof}
We only consider the case of $\gamma_{y_0}$ because
the proof for $\delta_{y_0}$ is similar.

Let $q\in H(y_0)$ be an arbitrary point. Then $q\in\omega_{q}$ and
by Lemma \ref{lAPS2} we have
\begin{equation}\label{eqMP1}
\omega_q({\alpha_{y_0}})=\omega({\alpha_{y_0}})\bigcap
h^{-1}{(q)}=\{\gamma_{q}\}.\nonumber
\end{equation}
Now we will show that $\mathfrak M^{+\infty}_{y_0}\subseteq
\mathfrak M^{+\infty}_{\gamma_{y_0}}$. Let $\{t_n\}\in \mathfrak
M^{+\infty}_{y_0}$, then there exists $q\in \omega_{y_0}$ such that
$\sigma(t_n,y_0)\to q$ and $t_n\to +\infty$ as $n\to \infty$.
Since the set $\omega_{\gamma_{y_0}}$ is conditionally compact,
the sequence $\{\pi(t_n,\gamma_{y_0})\}$ is relatively compact.
Let $z$ be a limit point of the sequence
$\{\pi(t_n,\gamma_{y_0})\}$, then there is a subsequence
$\{t_n^{'}\}\subseteq \{t_n\}$ such that
$\pi(t^{'}_{n},\gamma_{y_0})\to z$ as $n\to \infty$. On the other
hand $\sigma(t^{'}_{n},y_0)\to q$ as $n\to \infty$,
so $z\in \omega_{\gamma_{y_0}}\bigcap h^{-1}{(q)}\subseteq
\omega_{\alpha_{y_0}}\bigcap h^{-1}{(q)}=\{\gamma_{q}\}$, i.e.
$z=\gamma_{q}$. Since $\{\pi(t_n,\gamma_{y_0})\}$ is relatively
compact and $\gamma_{q}$ is its unique limit point, we have
$\pi(t_n,\gamma_{y_0})\to \gamma_{q}$ as $n\to \infty$, i.e.
$\{t_n\}\in\mathfrak M^{+\infty}_{\gamma_{y_0}}$.
The first statement then follows from Theorems \ref{tG3} and \ref{tG1}.

To prove (\ref{eqPM3}) it is sufficient to apply
Theorem \ref{thM1} (the second statement). The proof is complete.
\end{proof}

By Theorems \ref{thM2}, \ref{thShc6} and \ref{tS2}, we have

\begin{coro}\label{corP2_1} Under the conditions (C1)--(C3) if the point $y_0$ is quasi-periodic
(respectively, Bohr almost periodic, almost automorphic, Birkhoff recurrent), then:
\begin{enumerate}
\item the point $\gamma_{y_0}$ (respectively, $\delta_{y_0}$) has the same recurrent property as $y_0$;
\item the point $\alpha_{y_0}$ (respectively, $\beta_{y_0}$) is asymptotically quasi-periodic
(respectively, asymptotically Bohr almost periodic, asymptotically almost automorphic, asymptotically Birkhoff recurrent).
\end{enumerate}
If in addition $y_0$ is Lagrange stable, then the above items (i) and (ii) also hold for pseudo periodic and pseudo recurrent case.
\end{coro}

The following result has its independent interest, so we formulate it here in spite that it will not be used in what follows.

\begin{proposition}
Assume that the hypotheses of Lemma \ref{lAPS2} hold. Then for any $x$ satisfying $\gamma_q \le x < \omega_q(x_0)$,
we have
\[
\lim_{t\to +\infty} \rho (\pi(t,x),\pi(t,\gamma_{q})) =0.
\]
Similarly, for any $x$ satisfying $\omega_q(x_0) < x \le \delta_q$,
we have
\[
\lim_{t\to +\infty} \rho (\pi(t,x),\pi(t,\delta_{q})) =0.
\]
\end{proposition}

\begin{proof}
We only need to prove the case $\gamma_q \le x < \omega_q(x_0)$.
Take $\{t_n\}\in\mathfrak N_{q}^{+\infty}$. Since the set $\{\pi(t_n, \alpha_q):n\in \mathbb N\}$ is
conditionally compact by the condition (C2) and $\lim_{n\to\infty} \sigma(t_n,q)=q$,
the set $\{\pi(t_n, \alpha_q):n\in \mathbb N\}$
is precompact. But it follows from Lemma \ref{lAPS2} that $\omega_q(\alpha_q)=\{\gamma_q\}$, so
\begin{equation}\label{add1}
\lim_{n\to\infty} \rho (\pi(t_n,\alpha_q), \gamma_q)=0.
\end{equation}
On the other hand, since $\mathfrak N_{q}^{+\infty}\subseteq \mathfrak N_{\gamma_{q}}^{+\infty}$ by Theorem \ref{thM1},
we get
\begin{equation}\label{add2}
\lim_{n\to\infty} \rho (\pi(t_n,\gamma_q), \gamma_q)=0.
\end{equation}

Since $\gamma_q \le x < \omega_q(x_0)$, by the monotonicity it follows that
\[
\pi(t_n,\gamma_q) \le \pi(t_n, x) \le \pi(t_n,\alpha_q), \quad\hbox{for } n\in\mathbb N.
\]
Letting $n\to\infty$, we obtain from \eqref{add1} and \eqref{add2} that
\[
\lim_{n\to\infty}\rho(\pi(t_n,x),\gamma_q)=0.
\]
This together with \eqref{add2} yields that
\[
\lim_{n\to\infty}\rho(\pi(t_n,x),\pi(t_n,\gamma_q))=0.
\]
Since $\omega(\alpha_q)$ is uniformly stable by (C2) and $\Sigma^+_{\gamma_q}\subset\omega(\alpha_q)$,
the result now follows from Lemma \ref{lAPS1}. The proof is complete.
\end{proof}

\subsection{Convergence to Poisson stable motions}

In this subsection, we give some sufficient conditions which imply the convergence of all motions to
Poisson stable ones. This kind of convergence is fundamental in classical monotone dynamics (see, e.g. \cite{HS_2005,Smi_1995}).

\begin{theorem}\label{thM1_1}
Assume that (C1)--(C4) hold. For given $x_0\in \mathcal S$, let $y_0:=h(x_0)$.
If $y_0\in \omega_{y_0}$, then the following statements hold:
\begin{enumerate}
\item
 $\gamma_{y_0}\in \omega_{x_0}$;
\item the point $\gamma_{y_0}$ is comparable with $y _0$ by
character of recurrence and
\item
\begin{equation}\label{eqP0_1}
\lim\limits_{t\to +\infty}\rho(\pi(t,x_0),\pi(t,\gamma_{y_0}))=0.
\end{equation}
\end{enumerate}
The same result holds for $\delta_{y_0}$, i.e. items (i)--(iii) hold with $\gamma_{y_0}$ replaced by $\delta_{y_0}$.
\end{theorem}

\begin{proof}
We only need to prove the result for $\gamma_{y_0}$.
Under conditions (C1)--(C4) we have $\alpha_{y_0}\in \omega_{x_0}\bigcap
h^{-1}{(y_0)}$, and by Lemma \ref{lAPS2} we have
$\{\gamma_{y_0}\}=\omega_{\alpha_{y_0}}\bigcap h^{-1}{(y_0)}\subseteq
\omega_{x_0}\bigcap h^{-1}{(y_0)}$. This means that there exists a
sequence $t_n\to +\infty$ as $n\to \infty$ such that
$\sigma(t_n,y_0)\to y_0$ and $\pi(t_n,x_0)\to \gamma_{y_0}$. Since
$\mathfrak N_{y_0} \subseteq \mathfrak
N_{\gamma_{y_0}}$ by Theorem \ref{thM1}, we have
\begin{equation}\label{eqM1}
\rho(\pi(t_n,x_0),\pi(t_n,\gamma_{y_0}))\le
\rho(\pi(t_n,x_0),\gamma_{y_0}) +
\rho(\gamma_{y_0},\pi(t_n,\gamma_{y_0}))\to 0
\end{equation}
as $n\to \infty$. Now to finish the proof it is sufficient to apply
Lemma \ref{lAPS1}.
\end{proof}

\begin{coro}\label{com1}
Assume that the conditions of Theorem \ref{thM1_1} hold, then $\omega_{y_0}(x_0)$ is a singleton set and hence
we have
\[
\omega_{y_0}(x_0) = \{\alpha_{y_0}\}=\{\beta_{y_0}\}=\{\gamma_{y_0}\}=\{\delta_{y_0}\}.
\]
\end{coro}

\begin{proof}
Let $\{t_n\}\in \mathfrak N_{y_0}^{+\infty}$. Then $\{t_n\}\in \mathfrak N_{\gamma_{y_0}}^{+\infty}\bigcap
\mathfrak N_{\delta_{y_0}}^{+\infty}$ since $\gamma_{y_0}$ and $\delta_{y_0}$ are comparable with $y _0$ by
character of recurrence. So we have $(\pi(t_n,\gamma_{y_0}),\pi(t_n,\delta_{y_0})) \to (\gamma_{y_0},\delta_{y_0})$
as $n\to \infty$. On the other hand, it follows from \eqref{eqP0_1} for $\gamma_{y_0}$ and $\delta_{y_0}$ that
\[
\lim_{t\to +\infty}\rho(\pi(t,\gamma_{y_0}),\pi(t,\delta_{y_0}))=0.
\]
This enforces that $\gamma_{y_0}=\delta_{y_0}$.

Recall that
\[
\alpha_{y_0}\le x \le \beta_{y_0}\quad \hbox{for any } x\in \omega_{y_0}(x_0)
\]
and $(\{\gamma_{y_0}\}, \{\delta_{y_0}\}) = (\omega_{y_0}(\alpha_{y_0}), \omega_{y_0}(\beta_{y_0}))$
by Lemma \ref{lAPS2}.
On the other hand $\gamma_{y_0},\delta_{y_0}\in \omega_{y_0}(x_0)$ by Theorem \ref{thM1_1}, so we have
\[
\gamma_{y_0}\le x \le \delta_{y_0}\quad \hbox{for any } x\in \omega_{y_0}(x_0)
\]
by the monotonicity of the NDS and the invariance of $\omega(x_0)$. Thus it follows that
$\omega_{y_0}({x_0})=\{\gamma_{y_0}\}=\{\delta_{y_0}\}$. The proof is complete.
\end{proof}

\begin{coro}\label{corP2}
Under the conditions (C1)--(C4) if the point $y_0$ is stationary
(respectively, $\tau$-periodic, Levitan almost periodic, almost recurrent, Poisson stable), then:
\begin{enumerate}
\item the point $\gamma_{y_0}$ (respectively, $\delta_{y_0}$) has the same recurrent property as $y_0$;
\item the point $x_0$ is asymptotically stationary
(respectively, asymptotically $\tau$-periodic, asymptotically Levitan almost periodic,
asymptotically almost recurrent, asymptotically Poisson stable).
\end{enumerate}
\end{coro}

By Theorems \ref{thM1_1}, \ref{thM2} and Corollary \ref{corP2_1} we have

\begin{theorem}\label{com}
Assume that (C1)--(C4) hold, $x_0\in \mathcal S$ and $y_0:=h(x_0)\in Y$ is strongly Poisson stable.
Then the following statements hold:
\begin{enumerate}
\item
$\gamma_{y_0}\in \omega_{x_0}$;
\item the point $\gamma_{y_0}$ is strongly comparable with $y _0$ by
character of recurrence;
\item we have
\begin{equation}\label{eqP0_2}
\lim\limits_{t\to +\infty}\rho(\pi(t,x_0),\pi(t,\gamma_{y_0}))=0;
\end{equation}
\item the point $\gamma_{y_0}$ has the same recurrent property mentioned in Corollary \ref{corP2_1} as $y_0$
and the point $x_0$ has the same asymptotically recurrent property as $\alpha_{y_0}$.
\end{enumerate}
The same result holds for $\delta_{y_0}$, i.e. items (i)--(iv) hold with $\gamma_{y_0}$ replaced by $\delta_{y_0}$.
\end{theorem}

\begin{remark}\label{remP01}
\rm When the point $y_0$ is Bohr almost periodic,
then the results of Corollary \ref{com1} and Theorem \ref{com} coincide with
that of J. Jiang and X.-Q. Zhao \cite[Theorem 4.1]{JZ_2005}, i.e. $\omega(x_0)$ is isomorphic to
$\omega(y_0)$ and $x_0$ is asymptotically almost periodic.
\end{remark}

In Theorems \ref{thM1_1} and \ref{com}, we get that
the solutions will converge to the Poisson stable
ones by mainly monotone and uniformly stable conditions. Now we
give another criterion, adapted from W. Shen and Y. Yi \cite{SYY},
for the convergence by Lyapunov functions. To this end, let us
denote
\[
\tilde{\mathcal S}:=\{(x_1,x_2): x_1,x_2\in\mathcal S \hbox{ and } h(x_1)=h(x_2)  \}.
\]
and introduce

\begin{definition}\rm
A continuous function $L: \tilde{\mathcal S} \to \mathbb R_+$ is called a {\em Lyapunov function}
if it satisfies the following two conditions:
\begin{enumerate}
\item $L(x_1,x_2)=0$ if and only if $x_1=x_2$;

\item $L(\pi(t,x_1),\pi(t,x_2)) < L(x_1,x_2)$ for $x_1\neq x_2$ and $t>0$.
\end{enumerate}
The NDS \eqref{eqNDS1} is said to be {\em contracting} if it admits a Lyapunov function.
\end{definition}

\begin{theorem}[Global attracting property] \label{Lya}
Assume (C2) and that the NDS \eqref{eqNDS1} is contracting. For given $x_0\in \mathcal S$, if $y_0:=h(x_0)$
is Poisson stable, then the following statements hold for any $q$ satisfying $q\in\omega (q) \subset \omega(y_0)$:
\begin{enumerate}
\item $\omega_q(x_0)=\{\gamma_q^0\}$, a singleton set;

\item for any $x\in\mathcal S$ with $h(x)=q$, we have
\[
\lim_{t\to+\infty} \rho(\pi(t,x),\pi(t,\gamma_q^0)) =0.
\]
\end{enumerate}
\end{theorem}

\begin{proof}
(i) Firstly it follows from Theorem \ref{th01} that $\omega(x_0)$ is a nonempty conditionally compact invariant set,
and by Theorem \ref{th02} the NDS generates on $\omega(x_0)$ a two-sided dynamical system.
Assume that $x_1,x_2\in \omega_q(x_0)$, then by the uniform stability of $\omega(x_0)$ and the proof of Theorem \ref{th02}
we know that the trajectories on $\omega(x_0)$ is negatively distal, i.e. $\inf_{t\in\mathbb R_-}\rho(\pi(t,x_1),\pi(t,x_2))>0$.
By Lemma \ref{l3.9} and Corollary \ref{cor3.9.1} $\mathcal E^+_{y_0}$ is a group, so there exists $\{t_n\}\in\mathfrak N_{y_0}^{+\infty}$
such that $\pi(t_n,\cdot)|_{\omega_{y_0}(x_0)} \to e$ with $e$ being the identity element of $\mathcal E_{y_0}^{+}$.
In particular,
\[
\pi(t_n,x_i) \to x_i \quad \hbox{as } n\to\infty, i=1,2.
\]
Then it follows that
\[
L(x_1,x_2) = \lim_{n\to\infty} L (\pi(t_n,x_1),\pi(t_n,x_2)) < L(\pi(t_{n_0},x_1),\pi(t_{n_0},x_2))< L(x_1,x_2)
\]
for some $n_0\in \mathbb N$, a contradiction. Therefore, $\omega_q(x_0)$ is a singleton set.

(ii) Note that, like $\omega({x_0})$, $\omega(x)\neq\emptyset$ for any $x\in \mathcal S$
with $h(x)=q\in\omega(q)\subset\omega(y_0)$. We claim that $\omega_q(x)=\omega_{q}(x_0)$ for all $q$.
Indeed, if not, then similar to (i) $\omega_q(x)=\{\gamma_q\}$ is a singleton set with $\gamma_q\neq\gamma_q^0$.
Let $E=\omega(x_0)\bigcup\omega(x)$, then $E$ is conditionally compact and uniformly stable. By the same proof in (i)
we can conclude that $E_q$ is a singleton set, i.e. $\gamma_q=\gamma_q^0$, a contradiction.

If
\[
\lim_{t\to+\infty} \rho(\pi(t,x),\pi(t,\gamma_q^0)) \neq0,
\]
then this enforces that $\gamma_q=\gamma_q^0$ for all $q\in\omega(q)\subset\omega(y_0)$ does not hold.
This contradiction proves our result.
\end{proof}

\section{Applications}\label{S6}

\subsection{Ordinary differential equations} \label{sec6.1}
Let $\mathbb R^n$ be an
$n$-dimensional real Euclidean space with the norm $|\cdot|$. Let us
consider a differential equation
\begin{equation}
u'=f(t,u),\label{eq1.0.6}
\end{equation}
where  $f\in C(\mathbb{R}\times \mathbb R^n, \mathbb R^n)$. Along with
equation $(\ref{eq1.0.6})$ we consider its
$H$-class, i.e. the family of equations
\begin{equation}
v'=g(t,v),\label{eq1.0.7}
\end{equation}
where $g\in H(f):=\overline{\{f^{\tau}:\tau\in \mathbb{R}\}}$,
$f^{\tau}(t,u):=f(t+\tau,u)$ for all $(t,u)\in \mathbb{R}\times
\mathbb R^n$ and by bar we mean the closure in $C(\mathbb{R}\times
\mathbb R^n,\mathbb R^n)$, which is equipped with the compact-open topology and
this topology can be generated by the following metric
\begin{equation*}
d(f_1,f_2):=\sum_{k\ge
1}\frac{1}{2^k}\frac{d_k(f_1,f_2)}{1+d_k(f_1,f_2)},
\end{equation*}
where $d_k(f_1,f_2):=\sup\limits_{|t|\le k,|x| \le
k}|f_1(t,x)-f_2(t,x)|$.

\textbf{Condition (A1)}. The function $f$ is {\em regular}, that is,
for every equation (\ref{eq1.0.7}) the conditions of existence,
uniqueness and extendability on $\mathbb{R}_{+}$ are fulfilled.

Denote by $\varphi(\cdot,v,g)$ the solution of equation
(\ref{eq1.0.7}), passing through the point $v\in \mathbb R^n$ at the
initial moment $t=0$. Then the mapping
$\varphi:\mathbb{R}_{+}\times \mathbb R^n\times H(f)\to \mathbb R^n$ is well defined
and satisfies the following conditions (see, e.g. \cite{bro75,Sel_71}):
\begin{enumerate}
\item  $\varphi(0,v,g)=v$ for all $v\in \mathbb R^n$ and $g\in H(f)$;

\item  $\varphi(t,\varphi(\tau,v,g), g^{\tau})=\varphi(t+\tau,v,g)$
for all $ v\in \mathbb R^n$, $g\in H(f)$ and $t,\tau \in
\mathbb{R}_{+}$;

\item  the mapping $\varphi:\mathbb{R}_{+}\times \mathbb R^n\times H(f)\to \mathbb R^n$
is continuous.
\end{enumerate}

Denote by $Y:=H(f)$ and $(Y,\mathbb{R},\sigma)$ the shift dynamical system
on $Y$ induced from $(C(\mathbb R\times\mathbb R^n,\mathbb R^n),\mathbb R,\sigma)$,
i.e. $\sigma(\tau,g)=g^\tau$ for $\tau\in\mathbb R$ and $g\in Y$.
Then the equation (\ref{eq1.0.6}) generates a cocycle $\langle \mathbb
R^n,\varphi, (Y,\mathbb{R},\sigma)\rangle $ and an NDS $\langle (X,\mathbb{R}_{+},\pi),
(Y,\mathbb{R},\sigma), h\rangle$, where $X:= \mathbb R^n\times Y$,
$\pi:=(\varphi,\sigma)$ and $h=pr_2:X\to Y$.

Let $\mathbb R_{+}^{n}:=\{x=(x_1,\ldots,x_n)\in \mathbb R^{n}: x_i\ge
0 \hbox{ for } i=1,2,\ldots,n\}$. Then it defines a partial order on $\mathbb R^{n}$: $u\le v$
if and only if $v-u\in \mathbb R_{+}^{n}$.

\textbf{Condition (A2).} Equation (\ref{eq1.0.6}) is monotone. This
means that the cocycle $\langle \mathbb R^n,\varphi,$ $ (H(f),$
$\mathbb{R},$ $\sigma)\rangle$ (or shortly $\varphi$) generated by
(\ref{eq1.0.6}) is monotone: if $u,v\in \mathbb R^{n}$ and
$u\le v$ then $\varphi(t,u,g)\le \varphi(t,v,g)$ for all $t\ge 0$
and $g\in H(f)$.

Let $K$ be a closed cone in $\mathbb R^n$. The dual cone to $K$ is the
closed cone $K^{*}$ in the dual space $\big{(}\mathbb
R^{n}\big{)}^{*}$ of linear functions on $\mathbb R^n$, defined by
\begin{equation}\label{eqK01}
K^{*}:=\{\lambda \in \big{(}\mathbb R^{n}\big{)}^{*}: \lambda(x)=\langle
\lambda,x\rangle \ge 0\ \mbox{for any}\ x\in K\},\nonumber
\end{equation}
where $\langle \cdot , \cdot\rangle$ is the scalar product in
$\mathbb R^{n}$.

Recall (\cite{Smi_1987},\cite[ChV]{Smi_1995}) that a function $f\in
C(\mathbb R\times \mathbb R^{n},\mathbb R^n)$ is said to be
{\em quasi-monotone} if for any $(t,u),(t,v)\in \mathbb R\times \mathbb
R^n$ and $\phi \in \big{(}\mathbb R^{n}_{+}\big{)}^{*}$ we have: $u
\le v$ and $\phi(u)=\phi(v)$ imply $\phi(f(t,u))\le \phi(f(t,v))$.

\begin{lemma}\label{lK1} Let $f\in
C(\mathbb R\times \mathbb R^{n},\mathbb R^n)$ be a regular and
quasi-monotone function, then the following statements hold:
\begin{enumerate}
\item if $u\le v$, then $\varphi(t,u,f)\le \varphi(t,v,f)$ for any $t\ge
0$;
\item any function $g\in H(f)$ is quasi-monotone;
\item equation (\ref{eq1.0.6}) is monotone.
\end{enumerate}
\end{lemma}

\begin{proof} The first statement is proved in
\cite[ChIII]{HS_2005}.

Let $g\in H(f)$, then there exists a sequence $\{h_k\}\subset
\mathbb R$ such that $g(t,u)=\lim\limits_{k\to \infty} f^{h_k}(t,u)$
for any $(t,u)\in\mathbb R\times \mathbb R^n$. Let $u\le v$
($u,v\in\mathbb R^n$) and $\phi \in \big{(}\mathbb
R^{n}_{+}\big{)}^{*}$ such that $\phi(u)=\phi(v)$. Since $f$ is
quasi-monotone, we have
\begin{equation}\label{eqK1}
\phi(f(t+h_k,u))\le \phi(f(t+h_k,v)).
\end{equation}
Passing to the limit in (\ref{eqK1}) as $k\to \infty$ we obtain
that $g$ is quasi-monotone.

Finally, the third statement follows from the first and second
statements. The proof is complete.
\end{proof}

\begin{definition}\label{defOO1} \rm A solution $\varphi(t,u_0,f)$ of
equation (\ref{eq1.0.6}) is said to be:
\begin{enumerate}
\item[-] {\em (positively) uniformly stable}, if for
arbitrary $\varepsilon >0$ there exists $\delta
=\delta(\varepsilon)>0$ such that
$|\varphi(t_0,u,f)-\varphi(t_0,u_0,f)|<\delta$ ($t_0\in \mathbb R$,
$u\in\mathbb R^n$) implies
$|\varphi(t,u,f)-\varphi(t,u_0,f)|<\varepsilon$ for any $t\ge t_0$;
\item[-] {\em compact} on $\mathbb R_{+}$ if the set $Q:=\overline{\varphi(\mathbb
R_{+},u_0,f)}$ is a compact subset of $\mathbb R^{n}$, where $\varphi(\mathbb
R_{+},u_0,f):=\{\varphi(t,u_0,f): t\in \mathbb R_{+}\}$.
\end{enumerate}
\end{definition}


\begin{definition}[\cite{Che_2009,scher72,Sch85}]\label{defCS01} \rm
A solution $\varphi(t,u_0,f)$ of equation (\ref{eq1.0.6}) is called {\em comparable}
(respectively, {\em strongly comparable}, {\em uniformly comparable}) if the
motion $\pi(t,x_0)$ (here $x_0:=(u_0,f)$) is comparable (respectively, strongly comparable,
uniformly comparable) with $\sigma(t,f)$ by character of recurrence.
\end{definition}

Recall that a function $\varphi\in C(\mathbb R,\mathbb R^{n})$
is said to {\em possess the property $(A)$}, if the motion
$\sigma(\cdot,\varphi)$ generated
by $\varphi$ possesses this
property in the shift dynamical system $(C(\mathbb R,\mathbb R^{n}),
\mathbb R,\sigma)$. A function $f\in C(\mathbb R \times \mathbb R^{n},\mathbb
R^{n})$ is said to {\em possess the property $(A)$ in $t\in \mathbb{R}$
uniformly w.r.t. $u$ on every compact subset of $\mathbb R^{n}$} if
the motion $\sigma(\cdot,f)$ generated by $f$ possesses this property in the
shift dynamical system $(C(\mathbb{R}\times \mathbb R^{n},\mathbb
R^{n}),\mathbb{R},\sigma)$. In the quality of the property $(A)$ there can stand Lagrange stability,
periodicity, asymptotic periodicity, almost periodicity, asymptotic almost periodicity and so on.


If $x_0=(u_0,y_0)\in X=\mathbb R^n \times Y$ and $\alpha_{y_0}$ (respectively, $\gamma_{y_0}$) is the point from $X$
defined in Lemma \ref{lAPS2}, then we denote by $\alpha_{u_0}$ (respectively, $\gamma_{u_0}$) the point from $\mathbb R^n$
such that $\alpha_{y_0}=(\alpha_{u_0},y_0)$ (respectively, $\gamma_{y_0}=(\gamma_{u_0},y_0)$). The we have the following

\begin{theorem}\label{thA1}
Suppose that the following assumptions are fulfilled:
\begin{enumerate}
\item[-] the function $f\in C(\mathbb R\times \mathbb R^n,\mathbb R^n)$ is
positively Poisson stable in $t\in\mathbb R$ uniformly w.r.t. $u$ on every
compact subset from $\mathbb R^n$, i.e. there
exists a sequence $t_n\to +\infty$ as $n\to \infty$ such that
$f^{t_n}$ converges to $f$ in $C(\mathbb R\times \mathbb R^n,\mathbb
R^n)$;
\item[-] each solution $\varphi(t,u_0,f)$ of equation
(\ref{eq1.0.6}) is bounded on $\mathbb R_{+}$ and uniformly stable.
\end{enumerate}

Then under the conditions (A1)--(A2) the following statements hold:
\begin{enumerate}
\item[1.] for any solution $\varphi(t,u_0,f)$ of equation
(\ref{eq1.0.6}) there exists a solution $\varphi(t,\gamma_{u_0},f)$
of (\ref{eq1.0.6}) defined and bounded on $\mathbb R$ such that:
\begin{enumerate}
\item $\varphi(t,\gamma_{u_0},f)$ is a comparable solution of
(\ref{eq1.0.6});

\item $\lim\limits_{t\to +\infty}|\varphi(t,\alpha_{u_0},f)-\varphi(t,\gamma_{u_0},f)|=0$;
\end{enumerate}

\item[2.] if the function $f\in C(\mathbb R\times \mathbb
R^n,\mathbb R^n)$ is stationary (respectively, $\tau$-periodic,
Levitan almost periodic, almost recurrent, Poisson stable) in
$t\in \mathbb R$ uniformly w.r.t. $u$ on
every compact subset of $\mathbb R^n$, then
$\varphi(t,\gamma_{u_0},f)$ has the same recurrent property in $t$ and
hence the solution $\varphi(t,\alpha_{u_0},f)$ has the same asymptotic recurrent property in $t$.
\end{enumerate}
\end{theorem}


\begin{proof} Let $f\in C(\mathbb R\times \mathbb R^n,\mathbb R^n)$
and $(C(\mathbb R\times \mathbb R^n,\mathbb R^n),\mathbb R,\sigma)$
be the shift dynamical system on $C(\mathbb R\times \mathbb
R^n,\mathbb R^n)$. Denote by $Y:=H(f)$ and $(Y,\mathbb R,\sigma)$
the shift dynamical system on $H(f)$ induced from $(C(\mathbb R\times
\mathbb R^n,\mathbb R^n),\mathbb R,\sigma)$. Consider the cocycle
$\langle\mathbb R^n,\varphi,(Y,\mathbb R,\sigma)\rangle$ and NDS $\langle (X,\mathbb{R}_{+},\pi),
(Y,\mathbb{R},\sigma), h\rangle$, with $X= \mathbb R^n\times Y$,
$\pi=(\varphi,\sigma)$ and $h=pr_2:X\to Y$, generated by equation (\ref{eq1.0.6}) (see Condition (A1)).

Let $x_0:=(u_0,f)$. Since $\varphi(t,u_0,f)$ is bounded on $\mathbb R_{+}$, it follows from Lemma \ref{l3.1}
that $\Sigma^+_{x_0}$ is conditionally precompact; on the other hand, since $\varphi(t,u_0,f)$ is uniformly stable,
it follows from Corollary \ref{lUS1} and Remark \ref{open} that $\omega_{x_0}$ is uniformly stable.
Now to finish the proof it is sufficient to apply Theorems \ref{thM1} and \ref{thShc2.1}.
\end{proof}


\textbf{Condition (A3).} For any compact subset $K\subset \mathbb
R^n$ the function $f\in C(\mathbb R\times \mathbb R^n,\mathbb R^n)$
is bounded and uniformly continuous on $\mathbb R\times K$.

\begin{remark}\rm
Note that if a function $f\in C(\mathbb R\times \mathbb R^n,\mathbb
R^n)$ satisfies condition (A3), then the set $\{f^{h}:\ h\in \mathbb
R\}$ is precompact in $C(\mathbb R\times \mathbb R^n,\mathbb R^n)$
and vice versa.
\end{remark}

\begin{theorem}\label{thA2}
Suppose that the following assumptions
are fulfilled:
\begin{enumerate}
\item[-] the function $f\in C(\mathbb R\times \mathbb R^n,\mathbb R^n)$ is
strongly Poisson stable in $t\in\mathbb R$ uniformly w.r.t. $u$ on
every compact subset of $\mathbb R^n$;
\item[-] each solution $\varphi(t,u_0,g)$ of every equation
(\ref{eq1.0.7}) is bounded on $\mathbb R_{+}$ and uniformly stable.
\end{enumerate}

Then under the conditions (A1)--(A3) the following statements hold:
\begin{enumerate}
\item[1.] for any solution $\varphi(t,v_0,g)$ of equation
(\ref{eq1.0.7}) there exists a solution $\varphi(t,\gamma_{v_0},g)$
of (\ref{eq1.0.7}) defined and bounded on $\mathbb R$ such that:
\begin{enumerate}
\item $\varphi(t,\gamma_{v_0},g)$ is a uniformly comparable solution of
(\ref{eq1.0.7});
\item $\lim\limits_{t\to
+\infty}|\varphi(t,\alpha_{v_0},g)-\varphi(t,\gamma_{v_0},g)|=0$.
\end{enumerate}

\item[2.] if the function $f\in C(\mathbb R\times \mathbb
R^n,\mathbb R^n)$ is quasi-periodic (respectively, Bohr
almost periodic, almost automorphic, Birkhoff recurrent, pseudo recurrent,
uniformly Poisson stable) in $t\in \mathbb R$ uniformly w.r.t. $u$ on every compact subset of $\mathbb R^n$,
then the solution $\varphi(t,\gamma_{u_0},f)$ has the same recurrent property in $t$ and
hence the solution $\varphi(t,\alpha_{u_0},f)$ has the same asymptotic recurrent property in $t$.
\end{enumerate}
\end{theorem}


\begin{proof}
Note that under the condition (A3) the hull $H(f)$ is compact, so uniform comparability
is equivalent to strong comparability by Theorem \ref{thShc6}.
Then this theorem can be proved similarly as Theorem
\ref{thA1} using Lemma \ref{l3.1}, Corollary \ref{lUS1} and Theorems \ref{thM2}, \ref{tS2}.
\end{proof}

\subsection{Functional-differential equations with finite delay}

Let us first recall some notions and notations from \cite{hale}. Let
$r>0,\; C([a,b],\mathbb{R}^n)$ be the Banach space of all continuous
functions $\varphi:[a,b]\to \mathbb{R}^n$ equipped with the
$\sup$--norm. If $[a,b]=[-r,0]$, then we set $ \mathcal
C:=C([-r,0],\mathbb{R}^n)$. Let $\sigma\in\mathbb{R},\, A\ge 0$ and
$u\in C([\sigma-r,\sigma+A],\mathbb{R}^n)$. We will define $u_t\in
\mathcal C$ for any $t\in[\sigma,\sigma+A]$ by the equality
$u_t(\theta):=u(t+\theta),\;-r\le\theta\le0$. Consider a functional
differential equation
\begin{equation}\label{eq6.4.2}
\dot u=f(t,u_t),
\end{equation}
where $f:\mathbb{R}\times \mathcal C\to\mathbb{R}^n$ is continuous.

Denote by $C(\mathbb{R}\times \mathcal C,\mathbb R^n)$ the space of
all continuous mappings $f:\mathbb{R}\times \mathcal C\to
\mathbb R^n$ equipped with the compact-open topology. On the space
$C(\mathbb{R}\times \mathcal C,\mathbb R^n)$ is defined (see, e.g. \cite[ChI]{Che_2015} and \cite[ChI]{scher72})
a shift dynamical system $(C(\mathbb{R}\times \mathcal C,\mathbb
R^n),\mathbb R,\sigma)$, where $\sigma(\tau,f):=f^{\tau}$ for any
$f\in C(\mathbb{R}\times \mathcal C,\mathbb R^n)$ and
$\tau\in\mathbb R$ and $f^{\tau}$ is $\tau$-translation of $f$,
i.e. $f^{\tau}(t,\phi):=f(t+\tau,\phi)$ for any $(t,\phi)\in\mathbb
R\times \mathcal C$. Let us set
$H(f):=\overline{\{f^{\tau}: \tau\in\mathbb{R}\}}$.

Along with equation (\ref{eq6.4.2}) let us consider the family of
equations
\begin{equation}\label{eq6.4.3}
\dot v=g(t,v_t),
\end{equation}
where $g\in H(f)$.

\textbf{Condition (F1).} In this subsection, we suppose that
equation (\ref{eq6.4.2}) is regular, i.e. the conditions of
existence, uniqueness and extendability on $\mathbb R_{+}$ are
fulfilled.

\begin{remark}\label{remF1} \emph{Denote by $\tilde{\varphi}(t,u,f)$ the solution of
equation (\ref{eq6.4.2}) defined on $[-r,+\infty)$ (respectively,
on $\mathbb{R}$) with the initial condition $u\in
\mathcal C$. By $\varphi(t,u,f)$ we will denote below the
trajectory of equation (\ref{eq6.4.2}), corresponding to the
solution $\tilde{\varphi}(t,u,f)$, i.e. a mapping from
$\mathbb{R}_{+}$ (respectively, $\mathbb{R}$) into $\mathcal C$,
defined by $\varphi(t,u,f)(s):=\tilde{\varphi}(t+s,u,f)$ for any
$t\in\mathbb{R}_{+}$ (respectively, $t\in\mathbb{R}$) and $s\in
[-r,0]$. Below we will use the notions of ``solution" and ``trajectory" for equation (\ref{eq6.4.2})
as synonymous concepts. }
\end{remark}

It is well-known \cite{bro75,Sel_71} that the mapping $\varphi
:\mathbb R_{+}\times \mathcal C\times H(f) \to \mathcal  C$
possesses the following properties:
\begin{enumerate}
\item $\varphi(0,v,g)=v$ for any $v\in \mathcal C$ and $g\in
H(f)$; \item
$\varphi(t+\tau,v,g)=\varphi(t,\varphi(\tau,v,g),\sigma(\tau,g))$
for any $t,\tau\in\mathbb R_{+}$, $v\in \mathcal C$ and $g\in H(f)$;
\item the mapping $\varphi$ is continuous.
\end{enumerate}
Thus equation (\ref{eq6.4.2}) generates a cocycle $\langle \mathcal
C,\varphi,(Y,\mathbb R,\sigma)\rangle$ and an NDS $\langle (X,\mathbb R_{+},\pi)$, $(Y,\mathbb
R,\sigma),h\rangle$, where $Y:=H(f)$, $X:=\mathcal C\times Y$, $\pi :=
(\varphi,\sigma)$ and $h:=pr_2 :X\to Y$.

\begin{remark}\label{remQ1} \rm
Denote by $\mathcal B:=\{f\in C(\mathbb R\times \mathcal C,\mathbb R^{n}):$
$f$ is continuous in $t$ uniformly w.r.t $u$ on any bounded subset of $\mathcal C$ and
$f(B)$ is a bounded subset of $\mathbb R^n$ for any bounded subset
$B\subset \mathbb R\times \mathcal C\}$, equipped with the topology
of uniform convergence on every bounded subset of $\mathbb R\times
\mathcal C$. This topology can be defined by the following metric
\begin{equation}\label{eqQ2}
d(f_1,f_2):=\sum_{k\ge
1}\frac{1}{2^k}\frac{d_k(f_1,f_2)}{1+d_k(f_1,f_2)},
\end{equation}
where $d_k(f_1,f_2):=\sup\limits_{|t|\le k,||\phi||_{\mathcal C}\le
k}|f_1(t,\phi)-f_2(t,\phi)|$.
Note that
\begin{enumerate}
\item the metric space $(\mathcal B,d)$ is complete;
\item the subset $\mathcal B\subset C(\mathbb R\times \mathcal C,\mathbb
R^n))$ is translation invariant, i.e. $f^{h}\in \mathcal B$ for any
$f\in \mathcal B$ and $h\in\mathbb R$;
\item the mapping $\sigma : \mathbb R\times \mathcal B\to \mathcal
B$, defined by equality $\sigma(h,f):=f^{h}$ for any $(h,f)\in \mathbb
R\times \mathcal B$, is continuous.
\end{enumerate}
Thus on the space $(\mathcal B,d)$ is defined a shift dynamical
system $(\mathcal B,\mathbb R,\sigma)$.
\end{remark}


Let $\mathcal C_{+}:=\{\phi \in \mathcal C:\ \phi \ge 0,$ i.e.
$\phi(t)\ge 0$ for any $t\in [-r,0]\}$ be the cone of nonnegative
functions in $\mathcal C$. By $\mathcal C_{+}$ on the space
$\mathcal C$ is defined a partial order: $u\le v$ if and only if $v-u\in
\mathcal C_{+}$.

\textbf{Condition (F2).} Equation (\ref{eq6.4.2}) is monotone,
that is, the cocycle $\langle \mathcal C,\varphi,(H(f),$ $\mathbb
R,$ $\sigma)\rangle$ generated by (\ref{eq6.4.2}) possesses the
following property: if $u\le v$, then $\varphi(t,u,g)\le \varphi
(t,v,g)$ for any $t\ge 0$ and $g\in H(f)$.

Recall (see, e.g. \cite{Smi_1987}, \cite[ChV]{Smi_1995}) that a function $f\in
C(\mathbb R\times \mathcal C,\mathbb R^n)$ is said to be
{\em quasi-monotone} if $(t,u),(t,v)\in \mathbb R\times C$, $u \le v$, and
$u_{i}(0)=v_{i}(0)$ for some $i$, then $f_{i}(t,u)\le f_{i}(t,v)$.

\begin{lemma}\label{lQ11} Let $f\in
C(\mathbb R\times \mathcal C,\mathbb R^n)$ be a quasi-monotone
function, then the following statements hold:
\begin{enumerate}
\item if $u\le v$, then $\varphi(t,u,f)\le \varphi(t,v,f)$ for any $t\ge
0$;
\item any function $g\in H(f)$ is quasi-monotone;
\item $u\le v$ implies $\varphi(t,u,g)\le \varphi(t,v,g)$ for any $t\ge
0$ and $g\in H(f)$.
\end{enumerate}
\end{lemma}

\begin{proof} The first statement is proved in
\cite{Smi_1987}, \cite[ChV]{Smi_1995}.

Let $g\in H(f)$, then there exists a sequence $\{h_k\}\subset
\mathbb R$ such that $g(t,u)=\lim\limits_{k\to \infty}f^{h_k}(t,u)$
for any $(t,u)\in\mathbb R\times \mathcal C$. Let $u\le v$
($u,v\in\mathcal C$) and $u_i(0)=v_{i}$ for some $i$. Since $f$ is
quasi-monotone, we have
\begin{equation}\label{eqQ1}
f_{i}(t+h_k,u)\le f_{i}(t+h_k,v)
\end{equation}
and passing to limit in (\ref{eqQ1}) as $k\to \infty$ we obtain
that $g$ is quasi-monotone too.

Finally, the third statement follows from the first and second
statements, and this concludes the proof.
\end{proof}

\textbf{Condition (F3).} For any bounded subset $A\subset \mathcal
C$ the set $f(\mathbb R\times A)$ is bounded in $\mathbb R^n$.



\begin{lemma}\label{lQ1} Let $\varphi(t,u,f)$ be a bounded on $\mathbb
R_{+}$ solution of equation (\ref{eq6.4.2}), then under the
condition (F3) the set $\varphi(\mathbb R_{+},u,f)\subset \mathcal
C$ is precompact.
\end{lemma}

\begin{proof} This statement follows from the Lemmas 2.2.3 and 3.6.1 in
\cite{hale}.
\end{proof}

\begin{definition}\label{defQ1} \rm
A solution $\varphi(t,u_0,f)$ of equation (\ref{eq6.4.2}) is said to be:
\begin{enumerate}
\item[-] {\em (positively) uniformly stable}, if for
arbitrary $\varepsilon >0$ there exists $\delta
=\delta(\varepsilon)>0$ such that
$||\varphi(t_0,u,f)-\varphi(t_0,u_0,f)||_{\mathcal C}<\delta$
($t_0\in \mathbb R$, $u\in\mathcal C$) implies
$||\varphi(t,x,f)-\varphi(t,x_0,f)||_{\mathcal C}<\varepsilon$ for
any $t\ge t_0$;

\item[-] {\em compact on $\mathbb R_{+}$} if the set $Q:=\overline{\varphi(\mathbb
R_{+},u_0,f)}$ is a compact subset of $\mathcal C$, where by bar we mean
the closure in $\mathcal C$ and $\varphi(\mathbb
R_{+},u_0,f):=\{\varphi(t,u_0,f):\ t\in \mathbb R_{+}\}$.
\end{enumerate}
\end{definition}

\textbf{Condition (F4).} Every solution $\varphi(t,v,g)$ of every
equation (\ref{eq6.4.3}) is bounded on $\mathbb R_{+}$ and
uniformly stable.

Let $f\in \mathcal B$, $\sigma(t,f)$ be the motion (in the shift
dynamical system $(\mathcal B,\mathbb R,\sigma)$) generated by $f$,
$u_0\in \mathcal C$, $\varphi(t,u_0,f)$ be a solution of equation
(\ref{eq6.4.2}), $x_0:=(u_0,f)\in X:=\mathcal C\times H(f)$ and
$\pi(t,x_0):=(\varphi(t,u_0,f),\sigma(t,f))$ be the motion of
skew-product dynamical system $(X,\mathbb R_{+},\pi)$.

Like in ODE case, a solution $\varphi(t,u_0,f)$ of
equation (\ref{eq6.4.2}) is called
{\em comparable} (respectively, {\em strongly comparable} or {\em uniformly comparable}) if the
motion $\pi(t,x_0)$ is comparable (respectively, strongly comparable
or uniformly comparable) with $\sigma(t,f)$ by character of recurrence.

Applying the results from Sections \ref{S2}-\ref{S5} we can obtain
a series of results for functional differential equation
(\ref{eq6.4.2}). Below we formulate some of them.

\begin{theorem}\label{thQ1} Suppose that the following assumptions
are fulfilled:
\begin{enumerate}
\item[-] the function $f\in \mathcal B$ is
positively Poisson stable in $t\in\mathbb R$ uniformly w.r.t.
$u$ on every bounded subset from $\mathcal C$;
\item[-] each solution $\varphi(t,u_0,f)$ of equation
(\ref{eq6.4.2}) is bounded on $\mathbb R_{+}$ and uniformly stable.
\end{enumerate}

Then under the conditions (F1)--(F3) the following statements hold:
\begin{enumerate}
\item[1.] for any solution $\varphi(t,u_0,f)$ of equation
(\ref{eq6.4.2}) there exists a solution $\varphi(t,\gamma_{u_0},f)$
of (\ref{eq6.4.2}) defined and bounded on $\mathbb R$ such that:
\begin{enumerate}
\item $\varphi(t,\gamma_{u_0},f)$ is a comparable solution of
(\ref{eq6.4.2});
\item $\lim\limits_{t\to
+\infty}||\varphi(t,\alpha_{u_0},f)-\varphi(t,\gamma_{u_0},f)||_{\mathcal
C}=0$.
\end{enumerate}
\item[2.] if the function $f\in \mathcal B$ is stationary (respectively, $\tau$-periodic,
Levitan almost periodic, almost recurrent, Poisson stable) in $t\in
\mathbb R$ uniformly w.r.t. $u$ on every bounded subset
from $\mathcal C$, then $\varphi(t,\gamma_{u_0},f)$ has the same recurrent property and
hence $\varphi(t,\alpha_{u_0},f)$ has the same asymptotic recurrent property in $t$.
\end{enumerate}
Here the notations $\varphi(t,\alpha_{u_0},f)$ and $\varphi(t,\gamma_{u_0},f)$ have the similar meaning as in Section \ref{sec6.1}.
\end{theorem}
\begin{proof} Let $f\in \mathcal B$
and $(\mathcal B,\mathbb R,\sigma)$ be the shift dynamical system on
$\mathcal B$. Denote by $Y:=H(f)$ and $(Y,\mathbb R,\sigma)$ the
shift dynamical system on $H(f)$ induced from $(\mathcal B,\mathbb
R,\sigma)$. Consider the cocycle $\langle \mathcal
C,\varphi,(Y,\mathbb R,\sigma)\rangle$ generated by equation
(\ref{eq6.4.2}) (see Condition (F1)). Now to finish the proof of
Theorem it is sufficient to apply Lemma \ref{l3.1} and Theorems \ref{thM1},
\ref{thShc2.1}.
\end{proof}


\begin{theorem}\label{thQ2} Suppose that the following assumptions
are fulfilled:
\begin{enumerate}
\item[-] the function $f\in \mathcal B$ is
strongly Poisson stable in $t\in\mathbb R$ uniformly w.r.t. $u$ on
every bounded subset from $\mathcal B$;
\item[-] the set $H(f)$ is compact in $\mathcal B$.
\end{enumerate}

Then under the conditions (F1)--(F4) the following statements hold:
\begin{enumerate}
\item[1.] for any solution $\varphi(t,v_0,g)$ of every equation
(\ref{eq6.4.3}) there exists a solution $\varphi(t,\gamma_{v_0},g)$
of (\ref{eq6.4.3}) defined and bounded on $\mathbb R$ such that:
\begin{enumerate}
\item $\varphi(t,\gamma_{v_0},g)$ is a uniformly comparable solution of
(\ref{eq6.4.3});
\item $\lim\limits_{t\to
+\infty}||\varphi(t,\alpha_{v_0},g)-\varphi(t,\gamma_{v_0},g)||_{\mathcal
C}=0$.
\end{enumerate}

\item[2.] if the function $f\in \mathcal B$ is quasi-periodic (respectively, Bohr
almost periodic, almost automorphic, Birkhoff recurrent, pseudo recurrent,
uniformly Poisson stable) in $t\in \mathbb R$ uniformly w.r.t.
$u$ on every bounded subset from $\mathcal C$, then
$\varphi(t,\gamma_{u_0},f)$ has the recurrent property and hence
$\varphi(t,\alpha_{u_0},f)$ has the same asymptotic recurrent property in $t$.

\end{enumerate}
\end{theorem}


\begin{proof}
This theorem can be proved similarly as Theorem
\ref{thQ1} using Lemma \ref{l3.1} and Theorems \ref{thM2}, \ref{tS2}.
\end{proof}


\subsection{Parabolic systems}

Consider the following system of parabolic differential equations
\begin{equation}\label{eqP1}
\partial_{t}{w_j}=\nu_{j}\Delta w_{j}+ f_{j}(t,x,w_1,\ldots,w_n), \
j=1,\ldots,n
\end{equation}
in a smooth bounded domain $D\subset \mathbb R^{d}$, $d\le 3$, with
the Neumann boundary conditions. Here $\Delta$ is the Laplace
operator and $\nu_{j}$ are some positive constants and $f =
(f_1,\ldots,f_n)$ is a function satisfying certain conditions (specified
below). Let $\mathfrak B$ be the Banach space of continuous functions
$h:\overline{D}\times \mathbb R_{+}^{n}\to \mathbb R^{n}$ such that
all derivatives $\partial_{w_j}{h_i}$ are continuous on
$\overline{D}\times \mathbb R_{+}^{n}$ and $||h||_{\mathfrak B}<
\infty$, where
\begin{equation}\label{eqP02}
||h||_{\mathfrak B}:=\sup\{(1+|w|)^{-1}\sum\limits_{i\ge
1}|h_{i}(x,w)|+\sum\limits_{i,j\ge 1}|\partial_{w_j}{}h_i(x,w)|:
(x,w)\in \overline{D}\times \mathbb R_{+}^{n}\}.\nonumber
\end{equation}

Denote by $C(\mathbb R,\mathfrak B)$ the space of all continuous
functions $f:\mathbb R\to \mathfrak B$ equipped with the compact-open
topology and $(C(\mathbb R,\mathfrak B),\mathbb R,\sigma)$ the
shift dynamical system on $C(\mathbb R,\mathfrak B)$.

\textbf{Condition (P1).} The function $f=(f_1,\ldots,f_n)\in C(\mathbb
R,\mathfrak B)$, i.e. $f:\mathbb R\times
\overline{D}\times \mathbb R_{+}^{n}\to \mathbb R^{n}$ is continuous
in $t\in\mathbb R$ uniformly w.r.t. $(x,w)\in
\overline{D}\times \mathbb R_{+}^{n}$.

\textbf{Condition (P2).} The function $f=(f_1,\ldots,f_n)\in
C(\mathbb R,\mathfrak B)$ is positively Poisson stable in $t\in
\mathbb R$ uniformly w.r.t. $(x,w)\in \overline{D}\times
\mathbb R_{+}^{n}$, i.e. there exists a sequence $t_{k}\to +\infty$
as $k\to \infty$ such that $f(t+t_k,x,w)\to f(t,x,w)$ uniformly w.r.t.
$(t,x,w)\in [-l,l]\times \overline{D}\times \mathbb
R_{+}^{n}$ for any $l\in \mathbb N$.

\textbf{Condition (P3).} Each component $f_{i}(t,x,w)\ge 0$ for any $(t,x)\in
\mathbb R\times \overline{D}$, $w\in \mathbb R_{+}^{n}$ of the form
$w=(w_1,\dots,w_{i-1},0,w_{i+1},\ldots,w_n)$ ($i=1,\ldots,n$), and it
is cooperative, i.e.
$\frac{\partial{f_i(t,x,w)}}{\partial{w_j}}\ge 0$ for $i\not= j$
and $(t,x,w)\in \mathbb R\times\overline{D}\times\mathbb R_{+}^{n}$.

Along with (\ref{eqP1}) we will consider the problem
\begin{equation}\label{eqP2}
\partial_{t}{w_j}=\nu_{j}\Delta w_{j}+ g_{j}(t,x,w_1,\ldots,w_n) , \
j=1,\ldots,n
\end{equation}
in the domain $D$ with the Neumann boundary condition and
arbitrary function $g:=(g_1,\ldots,g_n)\in H(f)$, where
$H(f):=\overline{\{f^{h}:\ h\in \mathbb R\}}$, $f^{h}$ is
$h$-translation of $f\in C(\mathbb R,\mathfrak B)$ and by bar we mean
the closure in $C(\mathbb R,\mathfrak B)$.

\begin{remark}\label{remP1} \rm
Note that if the function $f\in C(\mathbb R,\mathfrak
B)$ possesses the property (P3), then every function $g\in H(f)$ possesses
the same property.
\end{remark}

Let $C(\overline{D},\mathbb R^{n})$ (respectively,
$C(\overline{D},\mathbb R^{n}_{+})$) be the space of all continuous
functions $f: \overline{D}\to \mathbb R^{n}$ (respectively,  $f:
\overline{D}\to \mathbb R^{n}_{+}$) and $V:=C(\overline{D},\mathbb
R^{n})$ equipped with the norm $||f||_{V}:=\max\{|f(x)|:\ x\in
\overline{D}\}$ (respectively, $V_{+}:=C(\overline{D},\mathbb
R^{n}_{+})$ with the same norm).
By $V_{+}$ it induces a partial order on $C(\overline{D},\mathbb R^{n})$:
$f\le g$ if and only if $g-f\in V_+$.

Under the conditions (P1)--(P3) it can be proved (see, for example,
\cite{Chu_2001}, \cite[ChIII]{Hen} and \cite[ChVII]{Smi_1995}) for
any $g\in H(f)$ and $v\in V_{+}$ that the problem (\ref{eqP2}) admits a
unique solution $\varphi(t,v,g)$ belonging to the space $C(\mathbb
R_{+},V_{+})$. Denote by $Y:=H(f)$ and $(Y,\mathbb R,\sigma)$ the
shift dynamical system on $H(f)$. From general properties of
solutions of (\ref{eqP2}) we have:
\begin{enumerate}
\item $\varphi(0,v,g)=v$ for all $v\in V_{+}$ and $g\in H(f)$;

\item  $\varphi(t,\varphi(\tau,v,g),g_{\tau})=\varphi(t+\tau,v,g)$
for $ v\in V_{+}$, $g\in H(f)$ and $t,\tau \in
\mathbb{R}_{+}$;

\item  the mapping $\varphi:\mathbb{R}_{+}\times V_{+}\times H(f)\to V_{+}$
is continuous;

\item  if $u\le v $, then $\varphi(t,u,g)\le \varphi(t,v,g)$ for all $t\ge 0$ and $g\in H(f)$.
\end{enumerate}
Therefore, the problem (\ref{eqP1}) generates a monotone cocycle $\langle V_{+},\varphi,
(Y,\mathbb{R},\sigma)\rangle $ and hence a monotone NDS
$\langle (X,\mathbb{R}_{+},\pi),\, (Y,\mathbb{R},\sigma), h\rangle
$, where $X:= V_{+}\times Y$, $\pi:=(\varphi,\sigma)$ and
$h:=pr_2:X\to Y$.


\begin{remark}\label{remP0} \rm
Since for $d\le 3$ the Sobolev space $H^2(D)$ is compactly embedded into
$C(\overline{D},\mathbb R^n)$, one can prove (see, e.g.
\cite{Hen}) that the set $\varphi(\mathbb R_{+},u_,f)$ is
precompact if it is bounded.
\end{remark}


Like in ODE case, the compactness, uniform stability of a solution $\varphi(t,u_0,f)$ to
(\ref{eqP1}) can be defined similarly, and the (strong, uniform) comparability of $\varphi(t,u_0,f)$
can also be defined similarly. Then we are in the position to state the following

\begin{theorem}\label{thP1}
Suppose that each solution $\varphi(t,u_0,f)$
of (\ref{eqP1}) is compact on $\mathbb R_{+}$ and
uniformly stable. Then under the conditions (P1)--(P3) the
following statements hold:
\begin{enumerate}
\item[1.] For any solution $\varphi(t,u_0,f)$ of equation
(\ref{eqP1}) there exists a solution $\varphi(t,\gamma_{u_0},f)$ of
(\ref{eqP1}) defined and compact on $\mathbb R$ such that:
\begin{enumerate}
\item $\varphi(t,\gamma_{u_0},f)$ is a comparable solution of
(\ref{eqP1});
\item $\lim\limits_{t\to+
\infty}||\varphi(t,\alpha_{u_0},f)-\varphi(t,\gamma_{u_0},f)||_{V}=0$.
\end{enumerate}

\item[2.] If the function $f\in C(\mathbb R,\mathfrak B)$ is stationary (respectively, $\tau$-periodic,
Levitan almost periodic, almost recurrent, Poisson stable) in $t\in
\mathbb R$, then the solution $\varphi(t,\gamma_{u_0},f)$ has the same recurrent property and hence
the solution $\varphi(t,\alpha_{u_0},f)$ has the same asymptotic recurrent property in $t$.
\end{enumerate}
Here the notations $\varphi(t,\alpha_{u_0},f)$ and $\varphi(t,\gamma_{u_0},f)$ have the similar meaning as in Section \ref{sec6.1}.
\end{theorem}

\begin{proof} Let $f\in C(\mathbb R,\mathfrak B)$
and $(C(\mathbb R,\mathfrak B),\mathbb R,\sigma)$ be the shift
dynamical system on $C(\mathbb R,\mathfrak B)$. Denote by $Y:=H(f)$
and $(Y,\mathbb R,\sigma)$ the shift dynamical system on $H(f)$
induced from $(C(\mathbb R,\mathfrak B),\mathbb R,\sigma)$. Consider
the cocycle $\langle V_{+},\varphi,(Y,\mathbb R,\sigma)\rangle$
generated by (\ref{eqP1}). Now to finish the proof it is
sufficient to apply Lemma \ref{l3.1} and Theorems \ref{thM1}, \ref{thShc2.1}.
\end{proof}


\begin{theorem}\label{thP2} Suppose that the following assumptions
are fulfilled:
\begin{enumerate}
\item[-] the set $H(f)$ is a compact subset of $C(\mathbb R,\mathfrak
B)$;
\item[-] the function $f\in C(\mathbb R,\mathfrak B)$ is strongly Poisson stable in $t\in\mathbb R$;
\item[-] each solution $\varphi(t,v_0,g)$ of every problem (\ref{eqP2}) is compact
on $\mathbb R_{+}$ and uniformly stable.
\end{enumerate}

Then under the conditions (P1)--(P3) the following statements hold:
\begin{enumerate}
\item[1.] For any solution $\varphi(t,v_0,g)$ of every problem (\ref{eqP2})
there exists a solution $\varphi(t,\gamma_{v_0},g)$ of (\ref{eqP2})
defined and compact on $\mathbb R$ such that:
\begin{enumerate}
\item $\varphi(t,\gamma_{v_0},g)$ is a uniformly comparable solution of
(\ref{eqP2});
\item $\lim\limits_{t\to+
\infty}||\varphi(t,\alpha_{v_0},g)-\varphi(t,\gamma_{v_0},g)||_{V}=0$.
\end{enumerate}

\item[2.] If the function $f\in C(\mathbb R,\mathfrak B)$ is quasi-periodic (respectively, Bohr
almost periodic, almost automorphic, Birkhoff recurrent, pseudo recurrent,
uniformly Poisson stable) in $t\in \mathbb R$, then the solution
$\varphi(t,\gamma_{u_0},f)$ has the same recurrent property and hence the
solution $\varphi(t,\alpha_{u_0},f)$ has the same asymptotic recurrent property in $t$.
\end{enumerate}
\end{theorem}

\begin{proof}
This theorem can be proved similarly to Theorem
\ref{thP1} using Lemma \ref{l3.1} and Theorems \ref{thM2}, \ref{tS2}.
\end{proof}


\section*{Acknowledgements}

This work is partially supported by NSFC Grants 11271151, 11522104, and the startup and
Xinghai Jieqing funds from Dalian University of Technology.

\end{document}